\documentclass[10pt]{amsart}
\usepackage{amsthm, amsmath,cleveref, amsfonts, amssymb, enumerate, textcmds, tensor, enumitem}
\usepackage[margin=1.25in]{geometry}
\usepackage{pst-node}
\usepackage{tikz-cd} 
\usepackage[toc,page]{appendix}
\usepackage{graphicx,tikz}
\usepackage{pinlabel}
\usepackage[colorinlistoftodos]{todonotes}

\newtheorem{thm}{Theorem}[section]
\newtheorem{lem}[thm]{Lemma}
\newtheorem{prop}[thm]{Proposition}

\newtheorem*{thma}{Theorem A}
\newtheorem*{thmb}{Theorem B}

\newtheorem*{thmc}{Theorem C}

\newtheorem*{queE}{Question E}
\newtheorem*{queF}{Question F}
\newtheorem*{cord}{Corollary D}
\newtheorem*{tech}{Technical Assumption}

\DeclareMathOperator{\codim}{codim}
\DeclareMathOperator{\rank}{Rank}

\DeclareMathOperator{\SR}{\mathcal{SR}}
\DeclareMathOperator{\Y}{\mathcal{Y}}
\DeclareMathOperator{\J}{\mathcal{J}}
\DeclareMathOperator{\Jb}{\mathbf{J}}
\DeclareMathOperator{\M}{\mathfrak{M}}

\DeclareMathOperator{\tr}{\textrm{trace}}

\theoremstyle{definition}
\newtheorem{defn}[thm]{Definition}
\newtheorem{remark}[thm]{Remark}
\newtheorem{example}[thm]{Example}
\begin{document}
\title[Harmonic Surfaces]{Spaces of harmonic surfaces in non-positive curvature}
\author{Nathaniel Sagman}
\address{Nathaniel Sagman: University of Luxembourg, 2 Av. de l'Universite, 4365 Esch-sur-Alzette, Luxembourg.} \email{nathaniel.sagman@uni.lu}
\begin{abstract}
Let $\mathfrak{M}(\Sigma)$ be an open and connected subset of the space of hyperbolic metrics on a closed orientable surface, and $\mathfrak{M}(M)$ an open and connected subset of the space of metrics on an orientable manifold of dimension at least $3$. We impose conditions on $M$ and $\M(M)$, which are often satisfied when the metrics in $\M(M)$ have non-positive curvature. Under these conditions, the data of a homotopy class of maps from $\Sigma$ to $M$ allows us to view $\mathfrak{M}(\Sigma)\times \mathfrak{M}(M)$ as a space of harmonic maps of surfaces. Using transversality theory for Banach manifolds, we prove that the set of somewhere injective harmonic maps is open, dense, and connected in the space of harmonic maps. We also prove some results concerning the distribution of harmonic immersions and embeddings in the space of harmonic maps.
\end{abstract}

\maketitle

\begin{section}{Introduction}
The theory of harmonic maps from surfaces is well developed and has proved to be a useful tool in geometry and topology. There are many broadly applicable existence theorems for harmonic maps, but, compared to other objects like minimal surfaces, their geometry is neither well behaved nor easy to understand. Locally, the most we can say about an arbitrary harmonic map from a surface is that, in a good choice of coordinates, up to small perturbations it agrees with an $n$-tuple of harmonic homogeneous polynomials (see the Hartman-Wintner theorem \cite{HW}). And in contrast, minimal maps are weakly conformal and hence have much nicer local properties. In this paper, we consider spaces of harmonic surfaces and study their \textit{generic} qualitative behaviour through transversality theory. The goal is twofold: to find nice properties shared by a wide class of harmonic maps, and to further develop the methods and analysis for future problems.

Throughout the paper, let $\Sigma$ be a closed and orientable surface of genus $g\geq 2$, and let $M$ be an orientable $n$-manifold, $n\geq 3$. Fixing integers $r\geq 2$ and $k\geq 1$, as well as  $\alpha, \beta \in (0,1)$ with $\alpha\geq \beta$, denote by $\M(\Sigma)$ an open and connected subset of the space of $C^{r,\alpha}$ hyperbolic metrics on $\Sigma$, and by $\M(M)$ an open and connected subset of the space of $C^{r+k,\beta}$ metrics on $M$. Set $C(\Sigma,M)$ to be the space of $C^{r+1,\alpha}$ mappings from $\Sigma\to M$. These all have $C^\infty$ Banach manifold structures.
\begin{defn}
A homotopy class $\mathbf{f}$ of maps from $\Sigma$ to $M$ is admissible if the subgroup $\mathbf{f}_*(\pi_1(\Sigma))\subset \pi_1(M)$ is not abelian.
\end{defn}
For the whole paper, we fix an admissible class $\mathbf{f}$ and assume the following.
\begin{tech}
For all $(\mu,\nu)\in \M(\Sigma)\times \M(M)$, there exists a unique harmonic map $f_{\mu,\nu}:(\Sigma,\mu)\to (M,\nu)$ in the class $\mathbf{f}$, and $f_{\mu,\nu}$ is a non-degenerate critical point of the Dirichlet energy functional. 
\end{tech}
The technical assumption is satisfied by a wide range of manifolds $M$ and families of metrics. The central example is that of a closed manifold $M$, with $\M(M)$ consisting of negatively curved metrics. In Section 2, we give more examples that are of interest in geometry and topology.

With this assumption, $\M=\M(\Sigma)\times \M(M)$ may be viewed as a space of harmonic surfaces inside $M$. More precisely, a result of Eells-Lemaire \cite[Theorem 3.1]{EL} (a consequence of the implicit function theorem for Banach manifolds) implies that around each pair of metrics $(\mu_0,\nu_0)\in \M$, there is a neighbourhood $U\subset \M$ such that the mapping from $U\to C(\Sigma,M)$ given by $$(\mu,\nu)\mapsto f_{\mu,\nu}$$ is $C^k$. By uniformization and conformal invariance of energy, the restriction to hyperbolic metrics on the source does not give up any information.  

Our first result concerns the notion of a somewhere injective map, which is originally from symplectic topology. 
\begin{defn}
A $C^1$ map $f:\Sigma\to M$ is somewhere injective if there exists a regular point $p\in \Sigma$ such that $f^{-1}(f(p))=\{p\}$. Otherwise, we say $f$ is nowhere injective.
\end{defn}
\begin{remark}
When the somewhere injective harmonic map $f$ has isolated singular set, or more generally the set $A(f)$ from Section 4 is connected, it is injective on an open and dense set of points. This follows from the Aronszajn theorem \cite[page 248]{Ar} (see also \cite[Theorem 1]{S}).
\end{remark}
We let $\mathfrak{M}^*\subset \M$ denote the space of metrics $(\mu,\nu)$ such that $f_{\mu,\nu}$ is somewhere injective.
\begin{thma}
The subset $\mathfrak{M}^*\subset \M$ is open, dense, and connected.
\end{thma}
\begin{remark}
A minimal map on a Riemann surface is nowhere injective if and only if it factors through a holomorphic branched cover \cite[Section 3]{GOR}, or the surface admits an anti-holomorphic involution that leaves the map invariant \cite[Theorem 1.1]{me}. Pseudoholomorphic maps from a surface to a symplectic manifold have the same property (see \cite[Chapter 2.5]{MS}). Harmonic maps, in contrast, do not have the same rigidity.
\end{remark}
\begin{remark}
These results for minimal surfaces, or more general branched immersions in the sense of \cite{GOR}, are proved using the factorization theorem of Gulliver-Osserman-Royden. The analogue of this theorem for harmonic surfaces is the subject of our previous paper \cite{me}. Fittingly, the factorization theorem for harmonic maps \cite[Theorem 1.1]{me} is a crucial ingredient in the proof of Theorem A.
\end{remark}
Secondly, we prove a set of results about the structure of the space of harmonic maps near somewhere injective maps. The somewhere injective condition, while not obviously significant, comes into play in transversality arguments used for spaces of minimal surfaces (see the paper of Moore \cite{M1} and the book that followed \cite[Chapter 5]{M2}) and pseudoholomorphic curves (see \cite[Chapter 3]{MS}). In some sense, nowhere injective surfaces play the same role as reducible connections in Yang-Mills moduli spaces. 
\begin{thmb}
Suppose $\dim M \geq 4$, and let $(\mu,\nu)$ be such that $f_{\mu,\nu}$ is somewhere injective and has isolated singularities. Then there exists a neighbourhood $U\subset \M$ containing $(\mu,\nu)$ such that the space of harmonic immersions in $U$ is open and dense. If $\dim M \geq 5$, then the space of harmonic immersions in $U$ is also connected.
\end{thmb}
\begin{thmc}
Suppose $\dim M \geq 5$, and let $(\mu,\nu)$ be such that $f_{\mu,\nu}$ is somewhere injective and has isolated singularities. Then there exists a neighbourhood $U\subset \M$ containing $(\mu,\nu)$ such that the space of harmonic embeddings in $U$ is open and dense. If $\dim M \geq 6$, then the space of harmonic embeddings in $U$ is also connected.
\end{thmc}
We obtain the following corollary.
\begin{cord}
If $\dim M \geq 4$, then any somewhere injective harmonic map with isolated singularities can be approximated by harmonic immersions. If $\dim M \geq 5$, then any such harmonic map can be approximated by harmonic embeddings. 
\end{cord}
In Section 6.5, we explain our use of the hypothesis that $f$ has isolated singularities and the possibility of removing it. We propose the following question.
\begin{queE}
Do the weak Whitney theorems hold for harmonic surfaces? That is,
\begin{enumerate}
    \item if $\dim M\geq 4$, is the space of harmonic immersions in $\M$ open and dense? If $\dim M\geq 5,$ is the space of harmonic immersions connected? 
    \item if $\dim M\geq 5$, is the space of harmonic embeddings in $\M$ open and dense? If $\dim M\geq 6,$ is the space of harmonic embeddings connected? 
\end{enumerate}
\end{queE}
The weak Whitney theorems \cite[Theorem 2]{Wh} state that a regular enough map between manifolds $g: X\to Y$ can be approximated by immersions if $\dim Y \geq 2\dim X$, and by embeddings if $\dim Y \geq 2\dim X + 1$. Combined with the Whitney trick, they yield the Whitney immersion theorem and the Whitney embedding theorem. One can give modern proofs of the weak theorems via transversality theory (see also Remark \ref{whitneyex}). The question has a positive answer for Moore's spaces of minimal surfaces \cite[Theorem 5.1.1 and 5.1.2]{M2}. It is also reasonable to ask about strengthening our hypothesis.
\begin{queF}
    Can we say more if we assume that the homotopy class $\mathbf{f}$ is incompressible or essential? 
\end{queF}
Such assumptions on $\mathbf{f}$ may give us more tools to probe Question E.

\begin{subsection}{Outline of paper and proofs.} In the next section, we define harmonic maps and associated Jacobi operators, and give examples of spaces of harmonic surfaces. These examples mostly require $\M(M)$ to be a space of non-positively curved metrics. We prove Proposition \ref{iso} to show that some positive curvature is allowed. In Section 3, we compute precise expressions near singularities for reproducing kernels for Jacobi operators. We proceed by constructing parametrices for some objects that resemble Green's operators.

The proof of Theorem A is contained in Sections 4 and 5. Section 4 is the reduction to a transversality lemma and Section 5 is the proof of that lemma. Since the details are technical, we explain the proof here. For disjoint open disks $P,Q\subset \Sigma$ and $\delta>0$, we set $$\mathcal{D}(P,Q,\delta)=\{(\mu,\nu)\in\mathfrak{M}: d_\nu(f_{\mu,\nu}(P),f_{\mu,\nu}(Q))>\delta, P,Q\subset \Sigma^{\SR}(f_{\mu,\nu})\},$$ where $\Sigma^{\SR}(f)$ is the super-regular set, to be defined in Section 4. For Theorem A, it is enough to prove that somewhere injective maps are open, dense, and connected in restriction to $\mathcal{D}(P,Q,\delta)$'s. We define a map $$\Theta : \Sigma^2 \times (P\times Q\times \mathcal{D})\to M^2\times M^2, \hspace{1mm} \Theta(r,s,p,q,\mu,\nu) = (f_{\mu,\nu}(r), f_{\mu,\nu}(s), f_{\mu,\nu}(p), f_{\mu,\nu}(q)).$$ If $\Theta(r,s,p,q,\mu,\nu)$ avoids the diagonal in $M^2\times M^2$, which we denote by $L$, then $f_{\mu,\nu}$ is somewhere injective. So, if we show that $\Theta$ is transverse to $L$, then the preimage has codimension $2\dim M \geq 6$. Since $\Sigma^2$ has dimension $4$, a general transversality result, Proposition \ref{targ}, shows that the projection of $\Theta^{-1}(L)$ to $\mathcal{D}$ is dense and connected. 

The real substance of the proof is to show that $\Theta$ is transverse to $L$. Given $f_{\mu,\nu},$ we set $\mathbf{F}$ to be the pullback bundle, with space of sections $\Gamma(\mathbf{F})$. A variation through harmonic maps starting at $f_{\mu,\nu}$, or just a harmonic variation, is any section of $\Gamma(\mathbf{F})$ obtained by choosing a path in $\mathfrak{M}$ based at $(\mu,\nu)$ and taking the derivative of the corresponding path in $C(\Sigma,M).$ We argue by contradiction and suppose that $\Theta$ is not a submersion at points that map to $L$. Invoking an existence result for reproducing kernels, this implies that at some pair of metrics $(\mu,\nu)$, there is a non-zero section $X:\Sigma \to \Gamma(\mathbf{F})$ such that for all harmonic variations $V\in\Gamma(\mathbf{F})$,
\begin{equation}\label{outline}
    \int \langle \mathbf{J}V,X\rangle dA_\mu=0.
\end{equation}
Above, $\mathbf{J}$ is the Jacobi operator for $f_{\mu,\nu}$. $X$ satisfies the Jacobi equation away from its singularities, and we show that these singularities can be resolved, making use of the local expressions from Section 3. $X$ thus extends to a global Jacobi field, which is our contradiction. 

To resolve the singularities, we vary the target metric on $M$ to find harmonic variations $V$ such that (\ref{outline}) gives us good information. One could also vary the source metric on $\Sigma$, but it shouldn't work too well: in some situations where the homotopy class $\mathbf{f}$ is compressible and the original harmonic surface $f_{\mu,\nu}(\Sigma)\subset M$ is totally geodesic, we will have $f_{\mu,\nu}(\Sigma) = f_{\mu+\dot{\mu},\nu}(\Sigma)$ for all admissible variations $\dot{\mu}$. Thus, we can't in general perturb away from a nowhere injective map.

The singularities of $X$ are at intersection points of harmonic disks $f(\Omega_1), f(\Omega_2)\subset M$, and we divide into cases: either the disks are tangential at the intersection point or they are not. When not only are they tangential but also $f(\Omega_1)=f(\Omega_2)$ and $f|_{\Omega_2}^{-1}\circ f|_{\Omega_1}$ is conformal, then our approach simply cannot work. To give one example of what can go wrong, if $f$ factors through a holomorphic branched covering map (in which case the homotopy class is compressible) and $\Omega_1$ and $\Omega_2$ are related by a covering transformation, then no matter how we vary the target metric, the harmonic maps will continue to factor in this way and identify the two sets. This is where we use the factorization theorem \cite[Theorem 1]{me} to say that the set of metrics giving rise to harmonic maps with this property can be removed from the space without disconnecting it. The tangential case is then settled using the super-regular condition (Section 4). For the non-tangential case, we choose variations supported in what we call ``fat cylinders" that give $\mathbf{J}V$ more support near some places than others.

In Section 6, we prove Theorems B and C, yet again by transversality theory. Right now, we explain only Theorem B, since Theorem C is a similar argument. We trivialize the complexified tangent bundle of $M$ and let $\sigma$ be the projection onto the $\mathbb{C}^n$ factor. Then we define a map $$\Psi: \tilde{\Sigma} \times \M \to \mathbb{C}^n, \Psi(p,\mu,\nu) =\sigma(f_z(p)),$$ where $z$ is the uniformizing parameter for the metric $\mu$ on the universal cover $\tilde{\Sigma}$, and $f_z=df\Big (\frac{\partial}{\partial z}\Big)$. We try to show that $\Psi$ is transverse to $\{0\}$ and the submanifold of $\mathbb{C}^n$ consisting of vectors whose real and imaginary parts are collinear. We achieve this near $(\mu,\nu)$ that yield somewhere injective harmonic maps with isolated singularities. By transversality theory, this gives Theorem B. 
\begin{remark}\label{whitneyex}
    Perhaps it's worth pausing to sketch a proof of Whitney's weak immersion theorem, which should explain the dimension bounds for Theorem B. If we replace $\mathfrak{M}$ with $C(\Sigma,M)$, then the map $\Psi$ is always transverse. Since the set of rank $1$ vectors in $\mathbb{C}^n$, say $Y,$ has codimension $n-1,$ by the Transversality Theorem for Banach manifolds, $\Psi^{-1}(Y)$ is a submanifold of dimension $n-1$. Since $\Sigma$ has dimension $2,$ the projection of $\Psi^{-1}(Y)$ to $C(\Sigma,M)$ should morally have codimension at least $(n-1)-2=n-3.$ So if $n\geq 4,$ the complement of this projection should be dense, and if $n\geq 5$ then the complement should be connected. These last statements are made precise and realized through Proposition \ref{targ}. The dimension bounds for Theorem C can be derived similarly.
\end{remark}

As in the proof of Theorem A, we suppose transversality fails, and then we find there must be a section $X:\Sigma\to \Gamma(\mathbf{F})$ that is annihilated by all $\mathbf{J}V$, where $V$ ranges over variations through harmonic maps.

The contradiction is different from that of Theorem A. We attach a particular holomorphic structure to the complexification $\mathbf{E}$ of $\mathbf{F}$. Using somewhere injectivity and a lemma of Moore \cite{M1}, we find there is an open set $\Omega$ on which $X$ is the real part of a holomorphic section of a special holomorphic line bundle $\mathbf{L}\subset \mathbf{E}.$ Making use of the isolated singularity condition, we analytically continue the ``imaginary part," so that $X$ is the real part of a globally defined meromorphic section $Z$ of $\mathbf{L}$.  From Section 3 we see that $Z$ has at most a simple pole at one point. We then check that the order of this section does not match up with the degree of $\mathbf{L}$. This final contradiction can also be seen through Riemann-Roch.

In an appendix, we state general transversality theorems for Banach manifolds and prove Proposition \ref{targ}, the general result we use to deduce density and connectedness results.
\end{subsection}

\begin{subsection}{Acknowledgements}
    in the case of $3$-manifolds, an argument for the first theorem is given in an unpublished manuscript of Vladimir Markovi{\'c} \cite{Ma}. The proof had a few small issues, which have been resolved here, and also more needs to be done for the argument to work in all dimensions. 

I thank Vlad for allowing me to absorb content from his manuscript, and for the many discussions that we had related to this project. This paper is intended to be independent and self-contained, and the reader should not have to consult \cite{Ma}. I have tried to keep similar notation.

I would also like to thank the anonymous referee for a careful reading, helpful suggestions that have improved the paper, and pointing out many typos and minor errors.

This paper was mostly written while I was a PhD student at the California Institute of Technology, and the first version was submitted while I was simultaneously a visiting student at the University of Oxford. At the time of completing this final version, I am a postdoc at the University of Luxembourg and funded by the FNR grant O20/14766753, $\textit{Convex Surfaces in Hyperbolic Geometry.}$
\end{subsection}

\end{section}

\begin{section}{Spaces of harmonic surfaces}
\begin{subsection}{Conventions}
Given two non-negative functions defined on some set $X$, we say $$f\lesssim g$$ if there exists a constant $C>0$ such that $f(x)\leq C g(x)$ for all $x\in X$. We define $f\gtrsim g$ similarly. If $X=\mathbb{R}$, and $f,g,h$ are functions from $X\to [0,\infty)$, we write $$f= g + O(h)$$ to mean $|f-g|\lesssim h$. Given Banach spaces $(B_i, ||\cdot||_i)$, equipped with an inclusion map $B_1\to B_2$, we write $$||V||_2\lesssim ||V||_1$$ to mean there is a uniform constant $C>0$ such that for all $V\in B_1$, we have $||V||_2\leq C||V||_1$. 
\end{subsection}

\begin{subsection}{Harmonic surfaces} Throughout, the space of $C^{n,\alpha}$ sections of a $C^{n,\alpha}$ vector bundle $V$ over $M$ is denoted $\Gamma(V)$. Here we are allowing $n=\infty$ and $n=\omega$ (real analytic). Given a map $f:\Sigma \to M$, we let $\mathbf{F}=f^*TM$ be the pullback of $TM$ over $\Sigma$. If $f$ is $C^{n,\alpha}$ then $\mathbf{F}$ is a $C^{n,\alpha}$ bundle. The derivative $df$ may be viewed as a section of the endomorphism bundle $\mathbf{T}=T^*\Sigma \otimes \mathbf{F}$. 

Fix Riemannian metrics $\mu,\nu$ on $\Sigma$ and $M$ respectively. By $\nabla^{\mathbf{F}}$ we denote the pullback connection of the Levi-Civita connection $\nabla^\nu$ on $M$. The Levi-Civita connection $\nabla^\mu$ on $T\Sigma$ dualizes to a connection on $T^*\Sigma$, and this tensors with $\nabla^\mathbf{F}$ to form a connection $\nabla^\mathbf{T}$ on the tensor product $\mathbf{T}$.
\begin{defn}
$f:(\Sigma,\mu)\to (M,\nu)$ is harmonic if the tension field $\tau=\tau(f,\mu,\nu)\in \Gamma(\mathbf{F})$ given by $$\tau = \tr_\mu \nabla^\mathbf{T}df$$ satisfies $\tau=0$.
\end{defn}

 Under the technical assumption, $f_{\mu,\nu}$ will be the unique harmonic map from $(\Sigma,\mu)\to (M,\nu)$. When working with fixed $(\mu,\nu)$ we sometimes write $f=f_{\mu,\nu}$. Let $TM^\mathbb{C}=TM\otimes \mathbb{C}$ denote the complexification of the tangent bundle of $M$ and $\mathbf{E}:=f^*TM^\mathbb{C}$ the pullback bundle. The connections $\nabla^\mathbf{F}$ and $\nabla^\mathbf{E}$ will be used quite often, so henceforward we condense $$\nabla := \nabla^\mathbf{F}, \hspace{1mm} \nabla^\mathbf{E}$$ when the context is clear. A section $W\in \Gamma(\mathbf{E})$ may be uniquely written as $W=\textrm{Re}(W)+i\textrm{Im}(W)$, where $\textrm{Re}(W), \textrm{Im}(W)\in \Gamma(\mathbf{F})$. 
 
  Any hyperbolic metric $\mu$ gives rise to a unique Riemann surface structure in which $\mu$ is conformal. Let $z=x+iy$ be a local complex parameter on an open subset of $\Sigma$ and set $$\frac{\partial}{\partial z}= \frac{1}{2}\Big ( \frac{\partial}{\partial x}-i\frac{\partial}{\partial y}\Big ), \hspace{1mm} \frac{\partial}{\partial \overline{z}}= \frac{1}{2}\Big ( \frac{\partial}{\partial x}+i\frac{\partial}{\partial y}\Big ).$$ When working in such a coordinate, we use the notation $\nabla_x = \nabla_{\frac{\partial}{\partial x}}$, $\nabla_y = \nabla_{\frac{\partial}{\partial y}}$, $\nabla_z = \nabla_{\frac{\partial}{\partial z}}$. We define local sections of $\Gamma(\mathbf{E})$ by $df\Big (\frac{\partial }{\partial x}\Big )=f_x$, $df(\frac{\partial }{\partial y})=f_y$, and $$df(\frac{\partial }{\partial z})=\frac{1}{2}df\Big (\frac{\partial }{\partial x} - i\frac{\partial }{\partial y}\Big )= \frac{1}{2}(f_x - if_y) = f_z.$$ One can check that
 \begin{equation}
     (f\circ h)_w = (f_z\circ h)h_w
 \end{equation}
for any holomorphic map such that $h(w)=z$. Therefore, the expression $f_zdz$ is a globally defined $\mathbf{E}$-valued $(1,0)$-form on $\Sigma$.

From a classical theorem of Koszul and Malgrange, the complex vector bundle $\mathbf{E}$ admits a unique holomorphic structure such that the $(0,1)$-component of the connection $\nabla^\mathbf{E}$ is the standard $\overline{\partial}$-operator. In the complex coordinate, the harmonic map equation reduces to $$\nabla_{\overline{z}}f_z = 0.$$ That is, $f_z$ is a local holomorphic section of $\mathbf{E}$. 
 \end{subsection}

 \begin{subsection}{The Jacobi operator} 
The equation $\tau=0$ arises as the Euler-Lagrange equation for the Dirichlet energy functional (defined over a suitable Sobolev space). The tension field may be seen as a map $$\tau : \M\times C(\Sigma,M)\to \Gamma(\mathbf{F}).$$ For $(\mu,\nu)$ fixed, the derivative in the $C(\Sigma,M)$ direction is the Jacobi operator \cite{EL}, which we are about to define. The Dirichlet energy is non-degenerate--or the technical assumption from the introduction is satisfied--if and only if the Jacobi operator has no kernel. 
 
 Let $f:(\Sigma,\mu)\to (M,\nu)$ be a $C^2$ (not necessarily harmonic) map and as before set $\mathbf{F}=f^*TM$. Let $\Delta$ denote the Laplacian induced by the connection $\nabla^{\mathbf{F}}$ and $R=R^M$ the curvature tensor of the Levi-Civita connection of $\nu$. The Jacobi operator $\mathbf{J}_f=\mathbf{J}:\Gamma(\mathbf{F})\to \Gamma(\mathbf{F})$ is defined $$\mathbf{J}V = \Delta V - \tr_\mu R(df,V)df \hspace{1mm} , \hspace{1mm} V\in \Gamma(\mathbf{F}).$$  If $z=x+iy$ is a local complex parameter and the conformal density is $\mu$, then 
\begin{equation}\label{18}
    \mathbf{J}V=-\nabla_x\nabla_x V - \nabla_y\nabla_y V - |\mu|^{-1}(R(f_x,V)f_x + R(f_y,V)f_y).
\end{equation}
 The Jacobi operator is a second order strongly elliptic linear operator and it is essentially self-adjoint in the sense that $$\int_\Sigma \langle \mathbf{J} V, W \rangle dA = \int_\Sigma \langle V,\mathbf{J}W\rangle dA$$ for all $V,W\in \Gamma(\mathbf{F})$. Above, recall that $\langle \cdot ,\cdot \rangle = \langle \cdot ,\cdot \rangle_\nu$ is the inner product on $\mathbf{F}$ induced by the metric $\nu$ on $M$. The integration over $\Sigma$ is with respect to the volume form $dA=dA_\mu$. 
 \begin{remark}
 The assumption $r\geq 3$ guarantees the coefficients of the operator are at least $C^2$. This is relevant for the regularity theory, and we use this implicitly throughout the paper.
 \end{remark}
\end{subsection}
\begin{subsection}{Calculus on vector bundles}
The following Banach spaces will come into play.
 \begin{itemize}
     \item For $1\leq p <\infty$, $(L^p(\mathbf{F}), ||\cdot||_p)$ is the space of $L^p$-bounded measurable sections of $\mathbf{F}$.
     \item For $k\in \mathbb{Z}_+$, $1\leq p<\infty$,  $(W^{k,p}(\mathbf{F}),||\cdot||_{k,p})$ is the Sobolev space of $k$-times weakly differentianble sections with $L^p$ derivatives with respect to the Levi-Civita connection.
       \item For $k\in \mathbb{Z}_+, \alpha\in (0,1)$, $(C^{k,\alpha}(\mathbf{F}), ||\cdot||_{k,\alpha})$ is the space of $k$-times differentiable sections whose $k^{th}$ derivatives are $\alpha$-H{\"o}lder.
       \item We can define these spaces in restriction to any open set $\Omega\subset \Sigma$. For $L^p(\mathbf{F}|_\Omega)$, we use the notation $||\cdot||_{p,\Omega}$, and likewise for the other Banach spaces.
 \end{itemize}
Above, if the vector bundle is only $C^{n,\alpha}$, we restrict $k\leq n$. For precise definitions and other basic facts, see \cite[Chapter 10]{Ni}. If we choose a different metric or connection on $\mathbf{F}$, the relevant Sobolev spaces are equal as sets of sections, and the identity map is bicontinuous. Thus, it is unambiguous to write $W^{k,p}(\mathbf{F})$ (and likewise for the other spaces), while not specifying the choices involved. 

Now we recall some results relevant to the Jacobi operator.
A Jacobi field is a section $V\in \Gamma(\mathbf{F})$ such that $\mathbf{J}V=0$. We again refer the reader to \cite[Chapter 10]{Ni}. From the basic elliptic theory, essential self-adjointness implies the following.
 \begin{prop}
Suppose there are no non-zero Jacobi fields. Then for every $p>1$ and $0<\alpha<1$, the operator $\mathbf{J}$ extends to a family of isomorphisms $\mathbf{J}: W^{2,p}(\mathbf{F})\to L^p(\mathbf{F})$, $\mathbf{J}: C^{2,\alpha}(\mathbf{F})\to C^{0,\alpha}(\mathbf{F})$. Each such isomorphism preserves the subspace of smooth sections.
\end{prop}
The result below is a consequence of the Weyl lemma for linear elliptic operators.
\begin{prop}
 Let $\Omega\subset \Sigma$ be open and $V$ be a measurable section over $\Omega$ such that $||V||_{p,\Omega}<\infty$ for some $1<p<\infty$. If $\mathbf{J}V=0$ weakly on $\Omega$, then $V$ is as regular as the bundle $\Gamma(\mathbf{F})$, and $\mathbf{J}V\equiv 0$ on $\Omega$.
\end{prop}

\end{subsection}

\begin{subsection}{Examples of spaces of harmonic surfaces}
Here we list some examples of manifolds $M$ and spaces of metrics $\M(M)$ satisfying the technical assumption.
\begin{example}
$M$ is a closed $n$-manifold that admits a metric of negative curvature, with $\M(M)$ consisting of negatively curved metrics.
\end{example}
In this case, Sampson proves in \cite[Theorem 4]{S} that there are no smooth Jacobi fields. In fact, he proves a more general result.
\begin{thm}[Sampson, Theorem 4 in \cite{S}]
Let $(M,\nu)$ be a closed Riemannian manifold with non-positive curvature. Suppose $f:(\Sigma,\mu)\to (M,\nu)$ is an admissible harmonic map and there is at least one point $p$ at which all sectional curvatures of $M$ at $f(p)$ are strictly negative. Then there are no non-zero Jacobi fields.
\end{thm}
Compactness of the target is not important.
\begin{example}
$M$ is not necessarily compact, all $\M(M)$ are negatively curved, and the induced mapping between the fundamental groups is irreducible. 
\end{example}
Irreducible means that after choosing basepoints $x\in \Sigma$, $y\in M$ and identifying the image subgroup of $$\mathbf{f}_*:\pi_1(\Sigma,x)\to \pi_1(M,y)$$ with a subgroup of isometries of $(\tilde{M},\tilde{\nu})$ covering $(M,\nu)$, the subgroup does not fix any point on the ideal boundary $\partial_\infty \tilde{M}$. The existence is due to Labourie \cite{La} (the argument is based on Donaldson's work \cite{D}), and Sampson's argument \cite[Theorem 4]{S} goes through to show that the Jacobi operator is an isomorphism. For some intuition, this class of examples includes admissible classes $f$ such that at least one simple closed curve is mapped by $\mathbf{f}_*$ to a class whose geodesic length is positive. Even more specific examples include convex cocompact manifolds of negative curvature, such as quasi-Fuchsian $3$-manifolds.

To demonstrate the level generality, we prove a slight extension of Sampson's result that allows for some positive curvature. 

\begin{defn}
A pair $(\mu,\nu)\in \M$ is $\mathbf{f}$-admissible if $(M,\nu)$ is non-positively curved and there exists a map $f\in \mathbf{f}\cap C(\Sigma,M)$ that is harmonic with respect to $(\mu,\nu)$ and a point $p\in \Sigma$ such that all sectional curvatures of $M$ are negative at $f(p)$.
\end{defn}
As discussed, $\mathbf{f}$-admissibility implies uniqueness of the harmonic map.

\begin{prop}\label{iso}
Suppose $(\mu,\nu)$ is $\mathbf{f}$-admissible, and let $\nu_n$ be a sequence of metrics converging to $\nu$. Furthermore, assume $f_j:(\Sigma,\mu)\to (M,\nu_j)$ is a sequence of harmonic maps converging to a harmonic map $f:(\Sigma,\mu)\to (M,\nu)$. Then $\mathbf{J}_{f_j}$ admits no non-trivial $C^2$ Jacobi fields for sufficiently large $j$.
\end{prop}
This gives another example of interest.
\begin{example}
A sufficiently small neighbourhood of an $\mathbf{f}$-admissible pair inside the space of all Riemannian metrics.
\end{example}
\begin{proof}
Since the harmonic maps $f_j,f$ are homotopic through $C^{n+1,\alpha}$ maps, the bundles $f_j^*TM=\mathbf{F}_j$ and $f^*TM=\mathbf{F}$ are isomorphic in the $C^{n+1,\alpha}$ category. We identify them all with the bundle $\mathbf{F}$. Under this identification, $\mathbf{F}$ inherits a family of Riemannian metrics $\langle \cdot , \cdot \rangle_j$ with corresponding Levi-Civita connections $\nabla_j$, as well as elliptic operators $\mathbf{J}_{f_j}$. Since $\nu_j\to\nu$ and $f_j\to f$, we have convergence of associated objects $\langle \cdot, \cdot\rangle_j\to \langle \cdot, \cdot \rangle:= \langle \cdot, \cdot \rangle_\nu$, $\nabla_j\to\nabla:=\nabla^\nu$, and $\mathbf{J}_{f_j}\to \mathbf{J}_f$ in the relevant topologies. Henceforth, rename $\mathbf{J}_j=\mathbf{J}_{f_j}$.

 One could write our Sobolev spaces more precisely as $$W^{k,p}(\mathbf{F},\mu,\nu,\nabla).$$ We write $W^{1,2}(\mathbf{F})$ to denote the usual Sobolev space for $\mathbf{F}$, and $W^{1,2}(\mathbf{F}_j)$ for $W^{1,2}(\mathbf{F}_j,\mu,\nu_j,\nabla_j)$ (and likewise for the $L^2$ spaces). In our notation, we set $||\cdot||_{2}$, $||\cdot||_{1,2}$ to be the norms for $\mathbf{F}$ and $||\cdot||_{2,j}$, $||\cdot||_{1,2,j}$ to be the norms for $\mathbf{F}_j$. From bicontinuity of the identity map between these Banach spaces, there exists $C_j\geq 1$ such that for all $V\in\Gamma(\mathbf{F})$, 
\begin{align*}
    &C_j^{-1}||V||_2 \leq ||V||_{2,j}\leq C_j||V||_2, \hspace{1mm} \textrm{and} \\
    &C_j^{-1}||V||_{1,2} \leq
    ||V||_{1,2,j} \leq C_j||V||_{1,2}.
\end{align*}
It is an easy exercise to show that $C_j\to 1$ as $j\to\infty$.

To prove the lemma, assume for the sake of contradiction that there is a subsequence (which we still denote $\nu_j$) and a family of non-zero sections $V_j\in C^2(\mathbf{F})$ such that $\mathbf{J}_jV_j=0$ and $||V_j||_2=1$. Necessarily, $$\int_\Sigma \langle \mathbf{J}_jV_j,V_j\rangle_j dA = 0.$$ Unravelling the definition of the Jacobi operator and integrating by parts, we obtain
\begin{equation}\label{29}
    \int_\Sigma |\nabla_j V_j|_j^2 - \int_\Sigma \langle \tr_\mu R^{\nu_j}(df_j, V_j)df_j, V_j\rangle_j dA =0.
\end{equation}.
\begin{remark}
We implicitly use that $\nabla_j$ is the Levi-Civita connection for $\nu_j$ to integrate by parts. If we tried to use the metric $\nu$, then some extra terms involving Christoffel symbols would appear.
\end{remark}
 Let $\sigma_j$ denote the maximum of $0$ and the largest sectional curvature of $M$ in the image of $f_j$. Then $$ \langle \tr_\mu R^{\nu_j}(df_j, V_j)df_j, V_j\rangle_j \leq \sigma_j|\tr_\mu(df_j)|_j^2|V_j|_j^2$$ pointwise. Convergence of $\nu_j\to \nu$ and $f_j\to f$ then implies  $$\langle \tr_\mu R^{\mu_j}(df_j, V_j)df_j, V_j\rangle_j \lesssim \sigma_j|V_j|_j^2.$$ Substituting into (\ref{29}) we see $$\int_\Sigma |\nabla_j V_j|_j^2\lesssim\sigma_j\int|V_j|_j^2 dA.$$ Again using convergence of $\nu_j\to\nu$, we see $\sigma_j \to 0$ as $j\to\infty$. Choosing $j$ large enough so that $C_j\lesssim 1$, and using $||V_j||_2=1$ we obtain
 \begin{equation}\label{30}
     \int_\Sigma |\nabla_j V_j|_j^2 dA \lesssim\sigma_j\to 0
 \end{equation}
 as $j \to \infty$. 

The above result gives uniform control on the $W^{1,2}(\mathbf{F}_j)$ norm of $V_j$, and hence we also have control on the $W^{1,2}(\mathbf{F})$ norm. Since $W^{1,2}(\mathbf{F})$ is reflexive, the Banach-Alaoglu theorem guarantees the existence of a subsequence along which $V_j$ converges weakly in $W^{1,2}(\mathbf{F})$ to a section $V\in W^{1,2}(\mathbf{F})$. By the Rellich lemma, we may pass to a further subsequence to obtain strong convergence in $L^2$, so that $||V||_2=1$. 

We now claim that $\nabla V=0$ in the sense of distributions, i.e., it is an almost everywhere constant field. Working in a conformal parameter $z=x+iy$ for $\mu$, we write out 
$$|\nabla_j V_j|^2 = \mu^{-1}\Big (|\nabla_{j,x}V_j|_\nu^2+|\nabla_{j,y}V_j|_\nu^2\Big )$$
and observe the linear maps $\nabla_{j,x}$, $\nabla_{j,y}$ converge strongly to $\nabla_x$ and $\nabla_y$ respectively in $\textrm{Hom}(W^{1,2}(\mathbf{F}),L^2(\mathbf{F}))$ with respect to the operator norm $||\cdot||_{OP}$.
Thus, $$||\nabla_x V_j||_2 \leq ||(\nabla_{x}-\nabla_{x,j})V_j||_2 + ||\nabla_{x,j}V_j||_2\leq ||\nabla_x -\nabla_{x,j}||_{OP}||V_j||_{1,2}+ C_j||\nabla_{x,j}V_j||_{2,j}.$$
Our observation above shows the first term decays to $0$ as $j\to \infty$. It follows from inequality (\ref{30}) that the second term tends to $0$ as well. Therefore $\nabla_x V_j\to 0$ strongly in $L^2$. By the same method we see $\nabla_{y}V_j\to 0$ strong as well. The claim follows.

We obtain a contradiction by arguing that $V=0$ on a set of positive measure. This would force $||V||_2=1$ to be impossible. This is essentially Sampson's observation in \cite[Theorem 4]{S}. From (\ref{29}) it follows that 
\begin{equation}\label{31}
    \int\langle \tr_\mu R^\mu (df, V) df, V\rangle dA = 0. 
\end{equation}
Since $(\mathbf{f},\nu)$ is an admissible pair, there is a point $p_0\subset \Sigma$ such that all sectional curvatures of $M$ are negative at $f(p)$. We extract a neighbourhood $\Omega\subset \Sigma$ on which $f$ is a regular embedding and there is a $c>0$ such that all sectional curvatures of $M$ at points in $f(\Omega)$ are bounded above by $-c$. Thus, from the non-positive curvature assumption on $\nu$, if $V$ is not zero almost everywhere, the left-hand side of (\ref{31}) is strictly negative. As discussed above, this is a contradiction, and so we are done.
\end{proof}
Finally, the results should hold for some more examples that we don't pursue here: manifolds with boundary (see \cite[Section 4]{EL}), non-orientable manifolds (Moore considers non-orientable minimal surfaces in \cite[Section 11]{M1}), and equivariant Anosov representations into Lie groups of non-compact type. For the analogue of the Eells-Lemaire result, applied to a suitable class of equivariant harmonic maps, we invite the reader to see \cite{Sl}. In these three cases, the only substantial missing ingredient is the factorization theorem \cite{me}. A version of the theorem should be true in these contexts, but it would take us too far afield in the current paper. 
\end{subsection}
 
\end{section}

\begin{section}{Reproducing kernels for the Jacobi operator}
Let $p\in \Sigma$ and $U\in \mathbf{F}_p$. We say that $X:\Sigma\backslash \{p\}\to \mathbf{F}$ is a zeroth order reproducing kernel for the Jacobi operator if, for all $W\in\Gamma(\mathbf{F})$, we have $$\langle W(p), U\rangle = \int_\Sigma \langle \mathbf{J} W, X \rangle dA.$$ For $V\in T_p\Sigma$, $X:\Sigma\backslash \{p\}\to \mathbf{F}$ is a first order reproducing kernel if, for all $W\in\Gamma(\mathbf{F})$, $$\langle (\nabla_{V} W)(p), U\rangle = \int_\Sigma \langle \mathbf{J} W, X \rangle dA.$$
In the proof of the main theorems, we need explicit expressions for the singularities of reproducing kernels. We compute these singularities by constructing the kernels directly. Independent of the work below, one can find general existence results in \cite[Section 3]{Ma}.
\begin{remark}
From the self-adjoint property, kernels satisfy $\mathbf{J}X=0$ away from the singularities. 
\end{remark}

\begin{subsection}{The parametrices}
Let $(\Omega,z)$ be a disk neighbourhood of $p$, and $\Omega'\subset\Omega$. In the local chart, extend the vector $U$ to a $C^2$ section $U(z)$. Let $\phi_n:\Omega\to [0,1]$ be a smooth function in $\Omega$ such that    
\begin{itemize}
\item $\phi_n$ has support in $\{|\zeta|\leq 1/n\}$,
\item $\phi_n$ integrates to $1$ in $\Omega'$, and 
\item $\phi_n$ converges in the sense of distributions to the Dirac delta $\delta_p$ as $n\to\infty$.
\end{itemize}
 Let $G(z,\zeta)$ be the ordinary Green's function on $\Omega'$, of the form $$G(z,\zeta) =\frac{1}{2\pi}\log |z-\zeta|^{-1}+r(z,\zeta),$$ where $r$ is smooth, and define a section $S_n$ in $\Omega'$ by $$S_n(z) = U(z)\int_{\Omega'}G(z,\zeta)\phi_n(\zeta) d\zeta\wedge d\overline{\zeta}.$$ Observe that $$\int_{\Omega'}G(z,\zeta)\phi_n(\zeta) d\zeta\wedge d\overline{\zeta}\to G(z,0) = \frac{1}{2\pi}\log |z|^{-1}+r(z,0)$$ as $n\to\infty$ with maximum regularity on $\Omega'\backslash\{p\}$ and in $L^p$ for all $1<p<\infty$. We then extend $S_n$ to a globally defined section of $\mathbf{F}$ with support in $\Omega$, in a way that $S_n$ converges as $n\to \infty$ in the $C^\infty$ sense on $\Sigma\backslash\{p\}$ to a section $S$ satisfying $$S(z) = \frac{1}{2\pi}\log|z|^{-1}U(z) + r(z,0)U(z).$$ By the defining properties of $G(z,\zeta)$, $$\frac{\partial^2}{\partial z\partial \overline{z}}\int_{\Omega'}G(z,\zeta) \phi_n(\zeta) d\zeta\wedge d\overline{\zeta} = \phi_n(z).$$ Using this, we compute that in $\Omega'$,
 \begin{align*}
     \nabla_z\nabla_{\overline{z}}S_n(z) &= (\nabla_z\nabla_{\overline{z}} U(z))\int_{\Omega'}G(z,\zeta)\phi_n(\zeta) + (\nabla_{\overline{z}} U)\frac{\partial}{\partial z}\int_{\Omega'}G(z,\zeta)\phi_n(\zeta) \\
     &+ (\nabla_{z} U)\frac{\partial}{\partial \overline{z}}\int_{\Omega'}G(z,\zeta)\phi_n(\zeta) + U(z)\phi_n(z).
 \end{align*}
 Set $\Phi_n^1 = \nabla_z\nabla_{\overline{z}} S_n(z) - U(z)\phi_n(z)$.
 \begin{lem}
 For all $1\leq p<2$, $\Phi_n^1$ converges along a subsequence in $L^p$ as $n\to\infty$.
 \end{lem}
 \begin{proof}
 It suffices to show that the three terms above all subconverge in $L^p$ near $0$. Since $G$ splits into a log term and a regular term, we only need show $L^p$-subconvergence for $$(\nabla_z\nabla_{\overline{z}} U(z))\int_{\Omega'}\log |z-\zeta|\phi_n(\zeta), \hspace{1mm} (\nabla_{\overline{z}} U)\frac{\partial}{\partial z}\int_{\Omega'}\log |z-\zeta|\phi_n(\zeta), \hspace{1mm} (\nabla_{z} U)\frac{\partial}{\partial \overline{z}}\int_{\Omega'}\log |z-\zeta|\phi_n(\zeta).$$
By the basic properties of $\phi_n$, $\int_{\Omega'}\log|z-\zeta|\phi_n(\zeta)\to \log |z|$ in $L^p$ as $n\to\infty$, so the first term $L^p$-converges to $\nabla_z \nabla_{\overline{z}}U(z)\log|z|$ in $\Omega'$, and away from $\Omega'$ our regularity assumptions give $L^p$ convergence. As for the second term, since $1/|z|$ is in $L^p(\Omega')$ for $1\leq p <2$, an application of dominated convergence shows it is equal to $$\frac{1}{2}(\nabla_{\overline{z}}U)\int_{\Omega'} \frac{\phi_n(\zeta)}{z-\zeta}d\zeta \wedge d\overline{\zeta}.$$ Taking $n\to\infty$, we have convergence for such $p$ to $$\frac{\nabla_{\overline{z}}U}{z}$$ in $\Omega'$, and nice convergence outside of $\Omega'$ (note we can make this continuous by choosing $U$ so that $\nabla_{\overline{z}}U=0$, but this is not necessary). The final term is handled similarly.
 \end{proof}
\end{subsection}

\begin{subsection}{The zeroth order kernel}
With the parametrices in hand, the remainder of the computation is a routine procedure. Complementary to $\Phi_n^1$, set $$\Phi_n^2= \frac{1}{\sigma^2}R(S_n,f_z)f_{\overline{z}}.$$ Let $\Phi_n=\Phi_n^1-\Phi_n^2$ and $\Psi_n=J^{-1}(\Phi_n)$. Here $R$ is the complexified curvature tensor of $M$ and $\sigma^2$ is the density of the conformal metric $\mu$ on $\Sigma_\mu$.
\begin{lem}
For every $1\leq p <2$, the sequence of norms of $||\Phi_n||_p$ is uniformly bounded. Moreover, for any $\alpha\in (0,1)$, $\Psi_n$ converges along some subsequence to a section $\Psi \in C^{0,\alpha}$.
\end{lem}
\begin{proof}
We showed above that $\Phi_n^1$ converges in $L^p$ to an $L^p$ section. As for $\Phi_n^2$, away from $0$ it converges locally uniformly to some $C^\infty$ section. Around $0$ we have the estimate $$\Phi_n^2\leq C\log|z|$$ for some $C>0$ and hence we have uniform $L^p$ bounds for all $p$.

Invoking Proposition \ref{iso}, $\Psi_n$ is uniformly bounded in $W^{2,p}(\mathbf{E})$ for any $p\in [1,2)$. The convergence result now follows from the Rellich-Kondrachov theorem, which gives a compact embedding from $W^{2,p}\to C^{0,\alpha}$ when $2-2/p> \alpha$.
\end{proof}
\begin{prop}
The zeroth order reproducing kernel is of the form 
\begin{equation}\label{log}
    X(z) = -\frac{1}{2\pi}\log|z|U(p) + B(z)
\end{equation}
where $B(z)$ is a $C^{0,\alpha}$ local section of $\mathbf{E}$ near $p$, for any $\alpha\in (0,1)$.
\end{prop}

\begin{proof}
Let $W\in \Gamma(\mathbf{E})$. In local coordinates, the complexified Jacobi operator is given by $$JW = \nabla_z \nabla_{\overline{z}}-\sigma^{-2}R(W,f_z)f_{\overline{z}}.$$ As $\mathbf{J}$ is essentially self-adjoint, 
\begin{align*}
   \int_{\Sigma} \langle \mathbf{J}W, S_n \rangle dA &= \int_\Sigma \langle W, \nabla_z \nabla_{\overline{z}}S_n\rangle dA - \int_\Sigma \langle W, \sigma^{-2}R(S_n,f_z)f_{\overline{z}}\rangle dA \\
   &= \int_\Sigma \langle W, \Phi_n \rangle dA + \int_\Sigma \langle W,  \mu_n \rangle dA \\ 
   &= \int_\Sigma \langle \mathbf{J}W, \Psi_n \rangle dA + \int_\Sigma \langle W, \mu_n \rangle.
\end{align*}
We reorganize this to $$\int_{\Sigma} \langle \mathbf{J}W, S_n \rangle dA- \int_\Sigma \langle \mathbf{J}W, \Psi_n \rangle dA = \int_\Sigma \langle W, \mu_n \rangle.$$
The term on the right tends to $\langle W, U(p)\rangle$ as $n\to \infty$. Meanwhile, passing to the subsequence from the previous lemma, the left-hand side converges to $$\int_\Sigma \langle \mathbf{J}W, S- \Psi\rangle$$ as $n \to \infty$. Therefore, $X=S-\Psi$, and the expression for $X$ is then derived from the local expression for $S$ stated above and the fact that $\Psi\in C^{0,\alpha}$ for any $\alpha\in (0,1)$.
\end{proof}
\begin{remark}
We have made no attempt to optimize the regularity of $B(z)$.
\end{remark}
\begin{remark}
If we change to a different (not holomorphic) coordinate $\varphi(z)=\varphi(x,y)$ with $\varphi(0)=0$, the expression may not be so simple, but we know it behaves asymptotically like a constant multiple of $\log |\varphi|^{-1}$. 
\end{remark}
\end{subsection}

\begin{subsection}{First order kernels}
We don't need explicit information for the singularity for the first order kernel, but we do need to know the rate at which it blows up. A calculation is given in \cite[Appendix A]{Ma}, that strongly uses that $\nabla_{\overline{z}}=\overline{\partial}$ for the Koszul-Malgrange holomorphic structure. Here we give a different method that works in more generality (and applicable for higher order kernels).

We find the first order kernel with respect to the tangent vector $\frac{\partial}{\partial z}$. Taking real and imaginary parts, we can then get any kernel. Extend the vector $U$ in a local trivialization so that $\nabla_z U(p)=0$. For $z\in\Omega'$, $\zeta\in\Omega$, we thus have a well-defined function $X(z,\zeta)$ such that $$\langle  V(z), U(z)\rangle = \int_\Sigma \langle \mathbf{J}V(\zeta), X(z,\zeta)\rangle dA(\zeta)$$ for all $V\in\Gamma(\mathbf{F})$. From the work above, $X(z,\zeta)$ takes the form $$X(z,\zeta) = \frac{1}{2\pi}U(\zeta)\log|z-\zeta|^{-1}+ B(z,\zeta),$$ where, for fixed $z$, $B(z,\zeta)$ is locally $C^{0,\alpha}$ away from $\{\zeta=z\}$. This function is not regular and in fact blows up on the diagonal (unless $U(z)=0$). Away from the diagonal, regularity in $z$ is the maximum of regularity of $U$ and the vector bundle: from the construction, we can choose $\mu_n$ and $S_n$ to vary nicely with $z$ for each $n$, and then we get the correct regularity in the limit.

 Observe $$\frac{\partial}{\partial z}\langle V(z), U(z)\rangle = \langle \nabla_z V(z), U(z)\rangle+\langle V(z),\nabla_z U(z)\rangle$$ in $\Omega'$. In terms of our integrals, differentiating under the integral via dominated convergence, we get $$\frac{\partial}{\partial z}\int_\Sigma \langle \mathbf{J}V(\zeta), X(z,\zeta)\rangle= \int_\Sigma \langle \mathbf{J}V(\zeta), \nabla_z X(z,\zeta)\rangle = \langle \nabla_z V(z), U(z)\rangle + \langle V(z),\nabla_z U(z)\rangle.$$ Setting $z=0$, we find that the first order kernel is given by $\nabla_z X(0,\zeta)$. From this we deduce the following.
\begin{prop}\label{zeroex}
In the complex coordinate $z$, the reproducing kernel is of the form 
\begin{equation}\label{overz}
    X(z) = \frac{1}{\pi z}U(p) + B(z)
\end{equation}
where $B(z)$ is a $C^{0,\alpha}$ local section of $\mathbf{E}$ near $p$, for any $\alpha\in (0,1)$.
\end{prop}
\end{subsection}

\end{section}

\begin{section}{Somewhere injective harmonic maps}
 As discussed earlier, Theorem A reduces to a transversality result, whose proof is given in the next section. Apart from a few things, the content of this section is adapted from \cite[Section 6]{Ma}. 
\subsection{Exceptional Riemann Surfaces} Our proof of Theorem A involves a ``super-regular" condition (defined below) that we would like to know is generic. The lemma below allows us to dismiss a class of metrics on which the condition fails. 
\begin{defn}
A Riemann surface $\Sigma$ is exceptional if either 
\begin{itemize}
    \item $\Sigma$ is a holomorphic branched cover of another Riemann surface of genus at least $2$ or
    \item $\Sigma$ admits an anti-holomorphic involution.
\end{itemize}
\end{defn}
The lemma below is a consequence of the factorization theorem established in \cite{me}.
\begin{lem}\label{factorization}
Suppose there is a pair of disks $\Omega_1,\Omega_2\subset \Sigma$ and a conformal diffeomorphism $h:\Omega_1\to\Omega_2$ such that $f\circ h =f$ on $\Omega_1$. Then the Riemann surface $\Sigma$ is exceptional.
\end{lem}
\begin{proof}
According to \cite[Theorem 1.1]{me}, if $h:\Omega_1\to\Omega_2$ is a holomorphic map between open subsets of $\Sigma$ such that $f\circ h = f$, then $f$ factors through a holomorphic branched covering map onto a surface $\Sigma_0$. If $\Sigma_0$ has genus less than two, then it is either a sphere or a torus. In both cases, the subgroup $$f_*(\pi_1(\Sigma))<\pi_1(M)$$ is abelian, which contradicts our assumption that the homotopy class $\mathbf{f}$ is admissible. If $h$ is anti-holomorphic, the result follows from Theorem 1.1 and the discussion in Section 4 of \cite{me}.
\end{proof}
This next result is well understood and one can find details in \cite[Appendix B]{Ma}. We set $\mathfrak{M}'(\Sigma)$ to be the set of metrics in  $\mathfrak{M}(\Sigma)$ giving rise to non-exceptional Riemann surfaces.
\begin{prop}\label{exceptional}
 $\mathfrak{M}'(\Sigma)$ is an open, dense, and connected subset of $\mathfrak{M}(\Sigma)$. 
\end{prop}
For ease of notation, we write $\mathfrak{M}=\mathfrak{M}'(\Sigma)\times \mathfrak{M}(M)$ instead of $\mathfrak{M}(\Sigma)\times \mathfrak{M}(M)$ throughout the rest of the paper.

\subsection{Super-regular points}
 Denote by $A(f)$ the set of $p\in \Sigma$ such that $f^{-1}(f(p))\subset \Sigma^{reg}(f)$. Given metrics $(\mu,\nu)$ and $p,q\in A(f)$, we say that the inner products $\mu(p)$ and $\mu(q)$ are conformal to each other via $f$ if the tangent planes $df(T_p\Sigma)$ and $df(T_q\Sigma)$ agree in $T_{f(p)}M$, and if the push forwards $f_*\mu(p)$ and $f_*\mu(q)$ are collinear.  
\begin{defn}
Given a map $f$, a point $p\in \Sigma$ is said to be super-regular if 
\begin{itemize}
\item $p\in A(f)$ and
\item if $f(p)=f(q)$, then $\mu(p)$ and $\mu(q)$ are not conformal to each other via $f$.
\end{itemize}
We denote the set of super-regular points for a map $f$ by $\Sigma^{\SR}(f)$. We define $\SR\subset \Sigma\times \M$ by $(p,\mu,\nu)\in \SR$ if $p\in \Sigma^{\SR}(f_{\mu,\nu})$. 
\end{defn}
\begin{prop}\label{dense}
Continuing to exclude the exceptional metrics from $\mathfrak{M}$, the set $\SR$ is open in $\Sigma\times \mathfrak{M}$ and $\Sigma^{\SR}(f)$ is open and dense in $\Sigma$. 
\end{prop}
We first treat $A(f)$ on its own. It is due to Sampson \cite[Theorem 3]{S} that the set of regular points of an admissible harmonic map is open and dense.
\begin{lem}
$A(f)$ is open and dense in $\Sigma$.
\end{lem}
\begin{proof}
Openness is obvious. As for density, suppose on the contrary that there is an open set $\Omega\subset \Sigma$ on which $f$ is regular but no point is in $A(f)$. By shrinking $\Omega$ we may assume $f|_\Omega$ is an embedding. We then find a small tubular neighbourhood $N\subset M$ of the submanifold $f(\Omega)$, in which the nearest point projection $\pi: N\to f(\Omega)$ is well defined. The set $S=f^{-1}(N)\subset \Sigma$ is then an open submanifold of $\Sigma$.

Let $g=\pi\circ f: S\to f(\Omega)$. If $y\in S$ is a singular point of $f$, then it is a singular point of $g$. By assumption, for each $u\in f(\Omega)$, the set $f^{-1}(u)$ contains a singular point of $g$. Thus, each point in $f(\Omega)$ is the image of a singular point $y\in S$ of the map $g$. This contradicts Sard's theorem. 
\end{proof}

\begin{proof}[Proof of Proposition \ref{dense}]
It is clear that both $\Sigma^{\SR}(f)$ and $\SR$ are open. It remains to prove $\Sigma^{\SR}(f)$ is dense. Note that the set $f^{-1}(f(x))$ is finite provided $x\in A$. Indeed, if $|f^{-1}(f(x))|= \infty$, then the closed set $f^{-1}(f(x))$ has an accumulation point, at which the rank of $df$ is necessarily strictly less than two (as $f$ cannot be an embedding near that point).

From the previous lemma, we are left to show that the conformality condition holds on a dense subset of $A(f)$. Arguing by contradiction, suppose that on an open subset $\Omega \subset A$ we have that for every $p\in\Omega$ there exists $q\in f^{-1}(f(p))$ such that $\mu(p)$ and $\mu(q)$ are conformal to each other via $f$. Given $p\in \Omega$, we have a finite number of disks $D_1,\dots, D_n$ with centers $p_i$ such that $f(p)=f(p_i)$ and with $\mu(p)$ and $\mu(p_i)$ conformal via $f$. We also assume $f$ is a regular embedding on $\overline{D_i}$ and $\overline{\Omega}$. Let $C_i\subset D_i$ be the closed set of points $x\in \overline{D_i}$ with the property that there exists $y\in\overline{\Omega}$ with $f(x)=f(y)$ and such that $\mu(x)$ is conformal to $\mu(y)$ via $f$. We claim that for at least one $i$, $C_i$ has non-empty interior. If not, then $$\mathcal{C}=\cup_i f(C_i)\cap f(\Omega)$$ has empty interior, for it is a finite union of closed nowhere dense sets. Choosing a sequence $(p_n)_{n=1}^\infty\subset \Omega\backslash (f^{-1}(\mathcal{C})\cap \Omega)$ converging to $p$, we can find another sequence $(q_n)_{n=1}^\infty\subset \Sigma\backslash( \cup_i \overline{D_i})$ with $f(p_n)=f(q_n)$ and $\mu(p_n)$ and $\mu(q_n)$ are conformal via $f$. Passing to a subsequence, the $q_n$ converge to some point $q\in\Sigma\backslash ( \cup_i D_i)$ such that $f(p)=f(q)$ and $\mu(p)$ and $\mu(q)$ are conformal via $f$. This contradicts our construction of the $D_i$, and so the claim is proved.

Relabelling so that $f(\Omega)$ and $f(D_1)$ intersect with non-empty interior, we can find open sets $\Omega_1\subset \Omega$ and $\Omega_2\subset D_1$ as well as a  diffeomorphism $h:\Omega_1\to\Omega_2$ such that $f\circ h = f$ on $\Omega_1$. The metrics $\mu$ and $h^*\mu$ are pointwise conformally equivalent on $\Omega$, and thus $h$ is a conformal map. This contradicts Lemma \ref{factorization}.
\end{proof}

\subsection{Proof of Theorem A}
 Denote by $\mathcal{J}$ the subset of nowhere injective maps.
\begin{lem}\label{Jclosed}
$\mathcal{J}$ is closed.
\end{lem}
\begin{proof}
If a somewhere injective map $f$ (which need not be harmonic) has an  injective point at $p$, meaning $f^{-1}(f(p))=\{p\}$, then there is an open set containing $p$ that consists only of injective points. Indeed, choose a disk $\Omega$ around $p$ on which $f$ is regular and $f|_\Omega$ is injective. If the claim fails, we can find $p_n\to p$ and $q_n\in \Sigma\backslash\Omega$ such that $f(p_n)=f(q_n)$. By compactness, the $q_n$ subconverge to a point $q$ at which $f(p)=f(q)$, a contradiction.

So, suppose $((\mu_n,\nu_n))_{n=1}^\infty\subset \J$ converges to $(\mu,\nu),$ and $f=f_{\mu,\nu}$ is somewhere injective with injective point $p$. There is a disk $\Omega$ around $p$ such that $f_{\mu_n,\nu_n}$ is injective on $\Omega$. Thus, there exists $p_n\in \Sigma\backslash\Omega$ such that $f_{\mu_n,\nu_n}(p_n)=f_{\mu_n,\nu_n}(p)$, and again we find a contradiction by extracting an accumulation point $q\neq p$.
\end{proof}

Let $P,Q\subset \Sigma$ be two disjoint open embedded disks in $\Sigma$. For $\delta>0$ we let $$\mathcal{D}(P,Q,\delta)=\{(\mu,\nu)\in\mathfrak{M}: d_\nu(f_{\mu,\nu}(P),f_{\mu,\nu}(Q))>\delta, P,Q\subset \Sigma^{\SR}(f_{\mu,\nu})\}.$$ It follows from Proposition \ref{dense} that $\mathcal{D}(P,Q,\delta)$ is an open subset of $\mathfrak{M}$. By Proposition \ref{dense} we also know that the set $\Sigma^{\SR}(f_{\mu,\nu})$ is dense in $\Sigma$, and therefore non-empty. Thus, each pair $(\mu,\nu)\in\mathfrak{M}$ is contained in $\mathcal{D}(P,Q,\delta)$ for some disks $P,Q$ and $\delta>0$. 
\begin{lem}\label{local}
Let $A\subset \mathfrak{M}$ be a subset. Suppose every pair $(\mu,\nu)\in\M$ has a neighbourhood $\mathcal{D}\subset \M$ such that $\mathcal{D}\backslash A$ is open, dense, and connected in $\mathcal{D}$. Then $\M\backslash A$ is open, dense, and connected in $\M$.
\end{lem}
The proof is trivial point-set topology and left to the reader. Thus, toward Theorem A it suffices to prove that every $\mathcal{D}\backslash \mathcal{J}$ is connected, where $\mathcal{D}$ ranges over connected components of $\mathcal{D}(P,Q,\delta)$. Henceforward we work on a single such component $\mathcal{D}$. Set $\Sigma^2=\Sigma\times \Sigma$, $M^2=M\times M$, and $\Y=\Sigma^2\times (P\times Q \times \mathcal{D})$. Define the map $\Theta:\Y\to M^2\times M^2$ by $$\Theta(r,s,p,q,\mu,\nu)=(f(r),f(s),f(p),f(q))$$ where we abbreviate $f=f_{\mu,\nu}$. As we have noted earlier, the map $(\mu,\nu)\mapsto f_{\mu,\nu}$ is $C^k$ and the evaluation map has the same regularity as $f$. We deduce $\Theta$ is $C^m$, where $m=\min \{k,n+1\}$.

Let $L$ be the diagonal $$L=\{((u,v),(u,v))\in M^2\times M^2\}.$$ The significance of $\Theta$ and $L$ is contained in the fact that $$\pi^{-1}(\mathcal{J})\subset \Theta^{-1}(L)$$ where $\pi:\Y\to\mathcal{D}$ is the projection onto the last factor. Indeed, suppose $(\mu,\nu)\in \mathcal{J}$. Then for each pair of points $(p,q)\in P\times Q$ there exists $(r,s)\not\in P\times Q$ such that $f(p)=f(r)$ and $f(q)=f(s)$. Thus $\Theta(r,s,p,q,\mu,\nu)\in L$.
\begin{remark}
$\pi^{-1}(\mathcal{J})$ also contains the set $L_{P,Q}\times \M$, where $L_{P,Q}$ is the intersection of $\Sigma^2\times (P\times Q)$ with the diagonal of $\Sigma^2\times \Sigma^2$. This set has codimension $4$ and will not play a role in any of our analysis.
\end{remark}
 Assuming the transversality lemma below, we prove Theorem A.
\begin{lem}\label{theta}
Let $(\mu,\nu)$ be a pair of metrics with $\mu$ not exceptional. Then for all $(r,s,p,q)\in \Sigma^2\times (P\times Q)$ such that $\Theta(r,s,p,q,\mu,\nu)\in L$, $\Theta$ is a submersion at that point. In particular, $\Theta$ is transverse to $L$ at such points.
\end{lem}
\begin{proof}[Proof of Theorem A]
    By Lemma \ref{Jclosed} we know that $\mathfrak{M}^*$ is open in $\mathfrak{M}$. If $(\mu,\nu)$ yields a somewhere injective harmonic map, then by openness there is nothing to do. According to Lemma \ref{factorization}, we can also dismiss pairs $(\mu,\nu)$ such that $\mu$ is exceptional. Henceforth fix $(r,s,p,q,\mu,\nu)$ such that $\mu$ is non-exceptional, and $f=f_{\mu,\nu}$ is nowhere injective. Via Lemma \ref{theta}, we can shrink the surrounding $\mathcal{D}(P,Q,\delta)$ so that $\Theta$ is transverse to $L$ on all of the corresponding $\mathcal{Y}$. To prove that $\mathcal{D}\backslash \mathcal{J}$ is dense and connected, we apply Proposition \ref{targ} with $A=P\times Q\times \mathcal{D}$, $X=\Sigma^2,$ $Y=M^2\times M^2,$ $f=\Theta,$ and $W=L.$ Proposition \ref{targ} applies since $d=\dim X =4$ and $k=\codim_{M^2\times M^2} W=2\dim M\geq 6,$ so $d-k\leq -2.$
\end{proof}
\end{section}

\begin{section}{The proof of transversality}
We give the proof of Lemma \ref{theta}.

\subsection{The derivative $d\Theta$} The deriative of $\Theta$ is a map $d\Theta:T\Y\to\mathbf{F}^2\times\mathbf{F}^2$. The tangent space $T\Y$ splits as $T(\Sigma^2\times P\times Q)\times T\M$. The restriction $d\Theta: T(\Sigma^2\times P\times Q)\times \{0\}\to \mathbf{F}^2\times \mathbf{F}^2$  is given by $$d\Theta(r,s,p,q,\mu,\nu,0) =df_r\times df_s\times df_p\times df_q$$ where $df_x$ denotes the derivative of $f=f_{\mu,\nu}$ at $x$. Since all four points are regular, the image of $d\Theta$ contains every quadruple of vectors $(Z_1,Z_2,Z_3,Z_4)\in \mathbf{F}^2\times \mathbf{F}^2$ that are tangent to the surface of $f(\Sigma)$ at the corresponding points.

For derivatives in the $\M$-coordinates, we leave the source metric $\mu$ fixed and vary the target metric. Let $\dot{\nu}\in T\M(M)$. By \cite[page 35]{EL}, the section $V\in\Gamma(\mathbf{F})$ defined by $$V=\frac{d}{dt}|_{t=0} f_{\mu,\nu+t\dot{\nu}}$$ satisfies $$\mathbf{J}V=\mathcal{G}(\dot{\nu}).$$  Here, $\mathcal{G}(\dot{\nu})$ is the derivative of the tension field in the $\dot{\nu}$-direction: $$\mathcal{G}(\dot{\nu}) = \frac{d}{dt}|_{t=0} \tau(\mu, \nu+t\dot{\nu},f_{\mu,\nu}).$$ Accordingly, $V$ is called a harmonic variation. It follows that $$d\Theta(0,0,0,0,0,\dot{\nu})=(V(r),V(s),V(p),V(q)).$$ 

Suppose $\Theta(r,s,p,q,\mu,\nu)\in L$. To simplify notation, we rename the points as $z_1=p$, $z_2=r$, $w_1=q$, $w_2=s$. The proof of Lemma \ref{theta} is another contradiction argument. Suppose the lemma is incorrect. Then, there are four vectors $Z_i\in \mathbf{F}_{z_i}$, $W_i\in\mathbf{F}_{w_i}$, not all of them zero, such that
\begin{itemize}
    \item $Z_i$ and $W_i$ are either zero or normal to the surface $f(\Sigma)$ at the points $f(z_i)$ and $f(w_i)$ respectively and
    \item for every $V$ such that $\mathbf{J}V=\mathcal{G}(\dot{\nu})$, the following holds: 
    \begin{equation}\label{21}
        \sum_{i=1}^2(\langle V(z_i),Z_i\rangle + \langle V(w_i),W_i\rangle) = 0.
    \end{equation}
\end{itemize}
 We now invoke reproducing formulas for the zeroth derivative. We showed the existence of reproducing kernels in Section 3. Adding up the four zeroth order reproducing kernels associated to the points $z_i,w_i$, we find a section $X:\Sigma\backslash \{z_1,z_2,w_1,w_2\} \to \mathbf{F}$ with maximum regularity and such that $$\sum_{i=1}^2(\langle W(z_i),Z_i\rangle + \langle W(w_i),W_i\rangle) = \int_\Sigma \langle \mathbf{J}W, X\rangle dA$$ for all $W\in \Gamma(\mathbf{F})$. We also record here that $\mathbf{J}X(p)=0$ for every $p\neq z_1,z_2,w_1,w_2$ and $X\in L^p(\mathbf{F})$ for every $p\geq 1$. $X$ is not identically equal to zero as we can certainly find sections $W\in \Gamma(\mathbf{F})$ such that the left-hand side above is not zero.
 On the other hand, from (\ref{21}) we conclude that $$\int_\Sigma \langle \mathbf{J}V, X\rangle dA = 0$$ for every harmonic variation $V$.
 
Stepping back for a moment, if $\dot{\nu}$ has support near $f(z_1)$, then the associated $\mathbf{J}V$ is supported near all preimages of $f(z_1)$. The kernel $X$ may have singularities at $z_1$ and $z_2$, while $X$ is smooth at the other preimages of $f(z_1)$. The tangent planes $df(T_{z_1}\Sigma)$ and $df(T_{z_2}\Sigma)$ are either tangential or span a $k$-plane for $k=3$ or $4$, and we find it convenient to treat the cases separately. In both cases, it is possible to choose $\dot{\nu}$ so that $\mathcal{G}(\dot{\nu})$ is negligible at $z_2$ but not so at $z_1$. In the tangential case, we extend the argument from \cite[Section 7]{Ma}. This is where the super-regular condition comes into play (and this is the only place it does). In this way, we can eliminate the singularity of $X$ at $z_1$. Repeating the procedure, but interchanging the roles of $z_1$ and $z_2$, we're able to show that $X$ is a global Jacobi field, which means $X\equiv 0$.
 
 \subsection{The time derivative of the tension field} We compute $\langle \Jb V, X\rangle$ in coordinates for a general variation $\dot{\nu}$.
 
 We let $(x_1,x_2)$ and $(u_1,\dots, u_n)$ denote local coordinates near $z_1\in \Omega$ and $f(z_1)\in M$ such that $z_1 = (0,0)$, $f(z_1) = (0,\dots, 0)$. Near $z_1$, the reproducing kernel $X$ can be expressed as a linear combination of the sections $\frac{\partial}{\partial f_j}=f^*\frac{\partial}{\partial u_j}$, $j=1,\dots, n$. We let $X^j$ denote the real valued functions on $\Omega$ such that $$X_k=\sum_{j=1}^nX^j\frac{\partial}{\partial f_j}.$$ In local coordinates on $\Sigma$ (not necessarily holomorphic), the tension field $\tau$ is given by $$\tau^\gamma=\tau^\gamma(f,\mu,\nu) = \mu^{ij}\Big (\frac{\partial^2 f^\gamma}{\partial x_i\partial x_j}- {}^\mu\Gamma_{ij}^k\frac{\partial f^\gamma}{\partial x_k} + {}^\nu\Gamma_{\alpha\beta}^\gamma(f)\frac{\partial f^\alpha}{\partial x_i}\frac{\partial f^\beta}{\partial x_j}\Big )$$ where $\gamma=1,\dots, n$ and, as in Section 3, we're using the Einstein summation convention. Here $\mu^{ij}$ are the components of the inverse of the metric tensor $\mu$. Let $\dot{\nu}$ be a variation of $\nu$ and set $\nu_t = \nu + t\dot{\nu}$. Recall we have defined $\mathcal{G}(\dot{\nu})=\frac{\partial}{\partial t}\tau(f_{\mu,\nu},\mu,\nu_t)|_{t=0}$. Since $\tau(f,\mu,\nu) = 0$, we see $$\langle \mathcal{G}(\dot{\nu}),X\rangle = \frac{d}{dt}|_{t=0}\langle \tau(f,\mu,\nu_t),X\rangle.$$ The only term that does not die upon taking the derivative is the term involving ${}^\nu\Gamma_{\alpha\beta}^\gamma(f)$. Thus, 
 \begin{equation}\label{idkk}
     \langle \mathcal{G}(\dot{\nu}),X\rangle = \frac{d}{dt}|_{t=0}\sum_{\alpha,\beta}\nu_{\alpha\beta}\mu^{ij}{}^{\nu_t}\Gamma_{\gamma\delta}^\alpha(f)\frac{\partial f^\gamma}{\partial x_i}\frac{\partial f^\delta}{\partial x_j}X^\beta.
 \end{equation}
 Set $${}^{\nu_t}\Gamma_{\gamma,\alpha\beta} = \frac{1}{2}(\nu^t_{\alpha\gamma,\beta} + \nu^t_{\gamma\beta,\alpha} -\nu^t_{\alpha\beta,\gamma})$$ where $\nu_{\alpha\beta,\delta}^t=\frac{\partial \nu_{\alpha\beta}^t}{\partial u_\delta}$ and $\nu_{t}^{\gamma\delta}$ denote inverse components of $\nu^t$.  Under this notation, the Christoffel symbols are computed by the well-known formula $${}^{\nu_t}\Gamma_{\alpha\beta}^\gamma = \sum_{\delta} \nu_t^{\gamma\delta} \cdot {}^{\nu_t}\Gamma_{\delta,\alpha\beta}.$$ Inserting back into (\ref{idkk}) yields
 \begin{equation}\label{50}
    \langle \mathcal{G}(\dot{\nu}),X\rangle = \frac{d}{dt}|_{t=0}\sum_{\alpha}\mu^{ij}{}^{\nu_t}\Gamma_{\gamma\delta,\alpha}(f)\frac{\partial f^\gamma}{\partial x_i}\frac{\partial f^\delta}{\partial x_j}X^\alpha= \sum_\gamma \mu^{ij}\dot{\Gamma}_{\alpha\beta,\gamma}\frac{\partial f^\alpha}{\partial x_i}\frac{\partial f^\beta}{\partial x_j}X^\gamma.
 \end{equation}
Here we are resuming the notation from Section 3, $$\dot{\Gamma}_{\alpha\beta}^\gamma = \lim_{t\to 0} \frac{\partial {}^{\nu_t}\Gamma_{\alpha\beta}^\gamma}{\partial t}.$$ We also record that 
\begin{equation}\label{51}
    \dot{\Gamma}_{\alpha\beta, \gamma} = \frac{1}{2}(\dot{\nu}_{\alpha\gamma,\beta} + \dot{\nu}_{\gamma\beta,\alpha} -\dot{\nu}_{\alpha\beta,\gamma}).
\end{equation}
\subsection{Tangential harmonic disks}
 Let $\Omega$ be a small neighbourhood of $z_1$ such that $f:\Omega\to M$ is an embedding. We let $(x_1,x_2)$ be conformal coordinates near $z_1$ and $(u_1,\dots, u_n)$ coordinates centered at $f(z_1)\in M$ such that 
\begin{itemize}
    \item $f(z_1)=(0,\dots,0)$,
    \item the (regular) surface $f(\Omega)$ is tangent to the plane $P=\{u_3=\dots=u_n=0\}$ at $f(z_1)$, and $u_i\circ f =x_i$ for $i=1,2$, and 
    \item $\nu_{jk}=\nu_{jk}=0$ and $\nu_{jj}=1$ for $k=1,2$ and $j=3,\dots, n$ when restricted to $P$ at $f(z_1)$.
\end{itemize}
Note that, as observed in Section 4, the set $f^{-1}(f(z_1))$ is finite. Set $$f^{-1}(f(z_1))=\{z_1,z_2,\dots, z_m \}.$$ For $\epsilon\in (0,1)$ small enough, we let $D(\epsilon)$ denote the disk of radius $\epsilon$ in the plane $P$, and let $D_\epsilon$ be the ball of radius $\epsilon$ in the $(u_1,\dots, u_n)$-coordinates centered at $0$. Since $z_k\in \Sigma^{reg}(f)$, we may choose $\epsilon$ so that $$f^{-1}(D_\epsilon)=\bigcup_{k=1}^m\Omega_k$$ where $\Omega_k=\Omega_k(\epsilon)$ is the corresponding neighbourhood of $z_k$. If we choose a variation $\dot{\nu}$ with support in $D_\epsilon$, then the induced variation of the pullback metric $f^*\nu$ is supported in $f^{-1}(D_\epsilon)$. If $\mathbf{J}V = \mathcal{G}(\dot{\nu})$, we will see that this implies $\mathbf{J}V$ is supported there as well, and we obtain 
\begin{equation}\label{intsum}
    \int_\Sigma \langle \mathbf{J}V,X\rangle dA = \sum_{k=1}^m \int_{\Omega_k}\langle \mathbf{J}V, X\rangle dA =0.
\end{equation}
Our proof of transversality of $\Theta$ involves analyzing each integral in the sum above. We split into cases: (i) the harmonic surfaces $f(\Omega_1)$ and $f(\Omega_2)$ are tangential at $f(z_1)$ and  (ii) they are not tangential. In each case, we pick a different variation of the target metric to find our contradiction.

We first treat case (i). For $\epsilon>0$ small enough, $$1\lesssim |df|\lesssim 1$$ on each $\Omega_k$. Here $|\cdot|$ is the operator norm. Since the surface $f(\Omega_k)$ is regular and proper, it follows that 
\begin{equation}\label{52}
   \epsilon^2 \lesssim \int_{\Omega_k} dA \lesssim \epsilon^2
\end{equation}
for all $k$. 

We now specify our variation. Let $Z_1^j$ denote the $j^{th}$ component of the vector $Z_1\in \mathbf{F}_{z_1}$ in the basis for $\mathbf{F}_{z_1}$ induced by our choice of coordinates. We let $\varphi_{\alpha\beta}=\varphi_{\beta\alpha}$ denote a set of real numbers such that $\varphi_{\alpha\beta}=0$ if at least one of $\alpha, \beta$ is greater than two. Let $\chi$ be a non-negative function of $(u_1,\dots, u_n)$ with support in $D_2$, equal to $1$ on $D_{1/2}$, and such that it has total integral $1$ with respect to the induced Euclidean area form on $P$. These conditions guarantee that, restricted to this plane, $\chi^\epsilon(u)=\epsilon^{-2}\chi(u/\epsilon)$ converges in the sense of distributions to the Dirac delta function as $\epsilon\to 0$. $\chi^\epsilon$ has compact support in $D(2\epsilon)$, and the product $\chi^\epsilon\varphi_{\alpha\beta}$ is equal to $\epsilon^{-2}\varphi_{\alpha\beta}$
on $D(\epsilon/2)$. Define $\dot{\nu}=\dot{\nu}(\epsilon)$ by $$\dot{\nu}_{\alpha\beta}(u) =\sum_{j=3}^n -2u_jZ_1^j\chi^\epsilon(u)\varphi_{\alpha\beta}.$$ We suppress the $\epsilon$ from our notation wherever possible. Referring back to (\ref{50}), we are interested in the variation of $\Gamma_{\alpha \beta,\gamma}$. For $\gamma\geq 3$, $$\dot{\Gamma}_{\alpha\beta,\gamma} = -\frac{1}{2}\dot{\nu}_{\alpha\beta,\gamma} = Z_1^\gamma\varphi_{\alpha\beta}\chi^\epsilon(u)+u_\gamma Z_1^\gamma\varphi_{\alpha\beta}\chi_\gamma^\epsilon(u),$$ and hence, on $D(\epsilon/2)$, 
\begin{equation}\label{epsilonfirstbound}
    |\dot{\Gamma}_{\alpha\beta,\gamma}| = \frac{1}{2}|\dot{\nu}_{\alpha\beta,\gamma}| = |Z_1^\gamma\varphi_{\alpha\beta}\chi^\epsilon(u)|=\epsilon^{-2}|Z_1^\gamma\varphi_{\alpha\beta}|\lesssim \epsilon^{-2}.
\end{equation}
For $\gamma=1,2$, 
\begin{equation}\label{epsilonsecondbound}
    |\dot{\Gamma}_{\alpha\beta,\gamma}| \lesssim \max_{\alpha\beta,\delta}|\dot{\nu}_{\alpha\beta,\delta}| \lesssim \epsilon|\nabla \chi^\epsilon|\lesssim \epsilon^{-2}.
\end{equation}
In any case, (\ref{epsilonfirstbound}) and (\ref{epsilonsecondbound}) show that we have a  $O(\epsilon^{-2})$ bound on $|\dot{\Gamma}_{\alpha\beta,\gamma}|$ for any choice of $\alpha,\beta,\gamma$. 

 The local coordinates $(x_1^1,x_2^1)$ near $z_1$ satisfy $$\frac{\partial f^\alpha}{\partial x_i^1}(z)=\delta_{i\alpha}.$$ Since $f(\Omega_2)$ is tangent to $\{u_3=0\}$ at $z_2$, we can choose coordinates $(x_1^2,x_2^2)$ near $z_2$ such that
 \begin{equation}\label{27}
      \frac{\partial f^\alpha}{\partial x_i^2}(z)=\delta_{i\alpha}+ O(\epsilon).
 \end{equation}
 Note that, by our restrictions on $\mu$, $\mu$ is no longer conformal in these coordinates.
\begin{remark}
When the two harmonic disks are equal, the $O(\epsilon)$ term is identically zero.
\end{remark}
Inserting these expressions into (\ref{idkk}) gives, near $z_1$, 
\begin{equation}\label{idk2}
    \langle \mathcal{G}(\dot{\nu}),X\rangle = \sum_\gamma \mu^{ij}\dot{\Gamma}_{ij,\gamma} X^\gamma
\end{equation}
and near $z_2$,  
 \begin{equation}\label{idk3}
     \langle \mathcal{G}(\dot{\nu}),X\rangle = \sum_\gamma \mu^{ij}\dot{\Gamma}_{\alpha\beta,\gamma}(\delta_{i\alpha}\delta_{j\beta}+O(\epsilon)) X^\gamma = \sum_\gamma \mu^{ij}\dot{\Gamma}_{ij,\gamma} X^\gamma + O(\epsilon)\sum_{\alpha,\beta,\gamma} \dot{\Gamma}_{\alpha\beta,\gamma}X^\gamma.
 \end{equation}
 
\subsection{Incompatible asymptotics}
The reproducing kernel $X$ is regular near each point $z_k$ when $k>2$. Trivially, $|\Jb V|\lesssim 1$ near $z_k$. Recalling (\ref{52}), we deduce $$\Big |\int_{\Omega_k} \langle \mathbf{J}V,X\rangle dA\Big | \lesssim \epsilon^{-2}$$ for $k>1,2$. For $k=1,2$, it may be the case that $X$ has a singularity near $z_1$ or $z_2$ (or both). We computed this singularity in Proposition \ref{zeroex}, the result being that in a trivialization near $z_1$, $$X(z) = \frac{1}{2\pi}\Big (\log \frac{1}{|z|}\Big )Z_k + B_1(z)$$ where $B_1(z)$ is a $C^{0,\alpha}$ local section of $\mathbf{F}$ near $z_1$, and $Z_k\in \mathbf{F}_{z_k}$ is the vector normal to the surface $f(\Sigma)$ at $f(z_1)$, defined above. Our coordinate is not conformal around $z_2$.

Here we are considering the zero vector to be normal. Recall we are assuming that for $k=1,2$, the patches $f(\Omega_k)$ are tangent to the plane $P$ at $f(z_k)$. $Z_k$ is normal to $P$ because $\nu_{1\gamma}=\nu_{2\gamma}=0$ for $\gamma\geq 3$. Incorporating these asymptotics into (\ref{idk2}) and (\ref{idk3}), we isolate that at $z_1$, 
\begin{equation}\label{z1}
    \langle \mathcal{G}(\dot{\nu}),X\rangle = \sum_{\gamma\geq 3}\mu^{ij}\dot{\Gamma}_{ij,\gamma}X^\gamma + \sum_{\delta=1,2}\mu^{ij}\dot{\Gamma}_{ij,\delta}X^\delta= \sum_{\gamma\geq 3} \mu^{ij}\dot{\Gamma}_{ij,\gamma}X^\gamma + O(\epsilon^{-2})
\end{equation}
and at $z_2$, $$\langle \mathcal{G}(\dot{\nu}),X\rangle = \sum_{\gamma\geq 3}\mu^{ij}\dot{\Gamma}_{ij,\gamma}X^\gamma(1+O(\epsilon)) + O(\epsilon^{-2}).$$
 We now use the fact that the restrictions of the metric $\mu$ at the points $z_1$ and $z_2$ are not conformal to each other via $f$. By the choice of local coordinates, this means the matrix $\mu^{ij}(z_1)$ is not a multiple of the matrix $\mu^{ij}(z_2)$, where both matrices are found by trivializing the pullback bundle $\mathbf{F}$ over $f^{-1}(D_\epsilon)$ using a trivialization of $D_\epsilon$. Furthermore, the two spaces of $2\times 2$ matrices orthogonal to $\mu^{ij}(z_1)$ and $\mu^{ij}(z_2)$ respectively (with respect to the Frobenius inner product) do not coincide. Thus, we can choose $\varphi_{ij}$ such that $$\sum_{i,j=1}^2\mu^{ij}(z_1)\varphi_{ij}=1$$ and $$\sum_{i,j=1}^2\mu^{ij}(z_2)\varphi_{ij}=0.$$ 
Taylor expanding (\ref{z1}), we see that in a trivialization around $z_1$,
\begin{align*}
    \langle \mathcal{G}(\dot{\nu}),X\rangle &= \sum_{\gamma\geq 3}\epsilon^{-2}\chi(z/\epsilon)Z_1^\gamma\Big ( \frac{1}{2\pi}\Big ( \log \frac{1}{|x|}\Big ) Z_1^\gamma +B^\gamma(x)\Big )+O(\epsilon^{-2})\\
    &= \sum_{\gamma\geq 3}\epsilon^{-2}\log|x|^{-1}\frac{\chi(z/\epsilon)}{2\pi}|Z_1^\gamma|^2 + O(\epsilon^{-2}).
\end{align*}
For $z$ near $z_2$, $$\sum_{i,j=1}^2\mu^{ij}(z)\varphi_{ij}\lesssim \epsilon$$ follows by Taylor expansion, and therefore
\begin{align*}
    \langle \mathcal{G}(\dot{\nu}),X\rangle &=\sum_{\gamma\geq 3}\mu^{ij}(Z_1^\gamma\varphi_{ij}\chi^\epsilon(u)+u_\gamma Z_1^\gamma\varphi_{ij}\chi_\gamma^\epsilon(u))X^\gamma(1+O(\epsilon))+O(\epsilon^{-2}) \\
    &\lesssim \sum_{\gamma\geq 3} \epsilon(Z_1^\gamma\chi^\epsilon(u)+u_\gamma Z_1^\gamma\chi_\gamma^\epsilon(u))(1+O(\epsilon))\Big (\Big ( \log \frac{1}{|x|}\Big ) Z_2^\gamma +B^\gamma(x)\Big )+O(\epsilon^{-2}) \\
    &\lesssim \sum_{\gamma\geq 3}\epsilon^{-1}\log\epsilon^{-1}|Z_1^\gamma Z_2^\gamma| +O(\epsilon^{-2}).
\end{align*}
Taking integrals yields 
\begin{align*}
    \int_{\Omega_1}\langle \mathbf{J}V,X\rangle dA &= \frac{\epsilon^{-2}}{2\pi}\sum_{\gamma\geq 3}\int_{\Omega_1}\Big ( \log \frac{1}{|x|}\Big )\chi(z/\epsilon)|Z_1^\gamma|^2 dA(x_1^1,x_2^1) + O(1) \\
      \int_{\Omega_2} \langle \mathbf{J}V,X\rangle dA &\lesssim \epsilon\log \epsilon^{-1}+O(1)\lesssim 1,
\end{align*}
and replacing back into (\ref{intsum}) gives
$$\int_\Sigma \langle \mathbf{J}V,X\rangle dA = \frac{\epsilon^{-2}}{2\pi}\sum_{\gamma\geq 3}\int_{\Omega_1}\Big ( \log \frac{1}{|x|}\Big )\chi(z/\epsilon)|Z_1^\gamma|^2 dA(x_1^1,x_2^1) + O(1).$$
Our standing assumption is that for all $\epsilon>0$, the left-hand side is equal to $0$. Therefore, $$\sum_{\gamma\geq 3}|Z_1^\gamma|^2 \frac{\epsilon^{-2}}{2\pi}\int_{\Omega_1}\Big ( \log \frac{1}{|x|}\Big )\chi(z/\epsilon) dA(x_1^1,x_2^1)= \Big|\sum_{\gamma\geq 3}\frac{\epsilon^{-2}}{2\pi}\int_{\Omega_1}\Big ( \log \frac{1}{|x|}\Big )\chi(z/\epsilon)|Z_1^\gamma|^2 dA(x_1^1,x_2^1) \Big | \lesssim 1.$$ In coordinates, $\Omega_1$ contains a ball of radius $\epsilon/2$ with respect to the Euclidean metric, and in such a ball $\chi(z/\epsilon)=1$. Thus, $$\frac{\epsilon^{-2}}{2\pi}\int_{\Omega_1}\Big ( \chi(z/\epsilon)\log \frac{1}{|x|}\Big ) dA(x_1^1,x_2^1)
\gtrsim\epsilon^{-2}\int_{\Omega_1}\Big ( \log \frac{1}{|x|}\Big ) dx_1^1 dx_2^1
\gtrsim \log(\epsilon^{-1}).$$
If there exists $\gamma\geq 3$ such that $Z_1^\gamma\neq 0$, this implies $$\log\epsilon^{-1}\lesssim 1,$$ which is nonsensical. This forces $Z_1^\gamma = 0$ for all $\gamma\geq 3$. Furthermore, recalling that $Z_1$ is normal to $f(\Omega_1)$ at $z_1$, we must have that $Z_1=0$ identically. This proves the following lemma.
\begin{lem}\label{extend1}
Suppose $f(\Omega_1)$ and $f(\Omega_2)$ are tangential at $f(z_1)$. Then $X$ extends smoothly over $z_1$.
\end{lem}

\subsection{Non-tangential harmonic disks}
We have essentially proved transversality $\Theta$, if we assume the images of the harmonic map are tangential at $z_1,z_2$, and at $w_1,w_2$. In this subsection, we consider other intersections. Namely, we prove the following.
\begin{lem}\label{extend2}
Suppose $f(\Omega_1)$ and $f(\Omega_2)$ are not tangential at $f(z_1)$. Then $X$ extends smoothly over $z_1$.
\end{lem} 
Equipped with this lemma, we can prove transversality with ease.
\begin{proof}[Proof of transversality of $\Theta$]  Assume $\Theta$ is not transverse, so that we have the section $X$ as in subsection 9.1. Applying Lemma \ref{extend1} or Lemma \ref{extend2}, depending on the circumstance, we see $X$ extends smoothly over the point $z_1$. Repeating this procedure with $Z_2,W_1$, and $W_2$ taking the role of $Z_1$, we can show it extends smoothly over those points as well. However, that means $X$ extends to a global Jacobi field, which can only occur if $X\equiv 0$. This is a contradiction.
\end{proof}
Moving toward the proof of Lemma \ref{extend2}, the proof of the tangential case does not immediately adapt because (\ref{27}) does not hold, and the super-regular condition can not be used effectively. To accommodate, we choose our variation differently. Instead of picking one supported in the ball $D_\epsilon$, we set $C_\epsilon = D(\epsilon)\times \{|u_j|<\epsilon^2: j=3,\dots, n\}$ and use  $$B_\epsilon = D_\epsilon \cap C_\epsilon.$$ The three-dimensional picture of this is the intersection of a ball with a fat cylinder. Similar to before, let $\Omega_k$ denote the connected components of $f^{-1}(B_\epsilon)$. Since $B_\epsilon$ is contained in $D_\epsilon$, regularity gives $$\int_{\Omega_k} dA \lesssim \int_{f^{-1}(D_\epsilon)} dA \lesssim \epsilon^2.$$ Following our previous approach, we choose real numbers $\varphi_{\alpha\beta}=\varphi_{\beta\alpha}$ to be specified later, but that can be non-zero only for $\alpha,\beta=1,2$. Let $\chi^\epsilon$ be exactly as before. Take $\omega_1$ to be a smooth function with support in $P\cap\{u_1^2+u_2^2< 4\}$ and such that $\omega_1 = 1$ on $P\cap\{u_1^2+u_2^2< 1\}$, and let $\omega_2:\mathbb{R}\to\mathbb{R}$ be a smooth function that is $1$ on $(-1,1)$ and $0$ off $(-2,2)$. Assuming $\epsilon < 1$, set $$\omega^\epsilon(u) = \omega_1(\epsilon^{-1}u_1,\epsilon^{-1}u_2,0,\dots,0)\prod_{k=3}^n\omega_2(\epsilon^{-2}u_k).$$
 $\chi^\epsilon\omega^\epsilon$ has compact support in $B_{2\epsilon}$, and $\chi^\epsilon\omega^\epsilon \epsilon^{-2}\varphi_{\alpha\beta}=\varphi_{\alpha\beta}$ in $B_\epsilon$. Define $\dot{\nu}=\dot{\nu}(\epsilon)$ by $$\dot{\nu}_{\alpha\beta}(u) = -2\sum_{j=3}^nu_jZ_1^j \chi^\epsilon(u)\omega^\epsilon(u)\varphi_{\alpha\beta}$$ and $\dot{\nu}_{\alpha\beta}=0$ for other $\alpha,\beta$. One can slightly adjust our previous computations to get $$\max_{\alpha,\beta} |\dot{\Gamma}_{\alpha\beta}^\gamma| \lesssim \epsilon^{-2}.$$ Reusing our previous notation, one can now show
$$\int_\Sigma \langle \mathbf{J}V,X\rangle dA = \sum_{k=1}^2\int_{\Omega_k} \nu_{\alpha\beta}(f) \mathcal{G}^\alpha(\dot{\nu})X_k^\beta dA + O(1).$$
We aim to find asymptotics for both integrals on the right-hand side above. The key observation in bounding the integral over $\Omega_2$ is that, since $f(\Omega_2)$ is not wholly contained in $B_\epsilon$, we have a stronger area estimate.
\begin{lem}
The area of $\Omega_2$ satisfies $$\textrm{Area}(\Omega_2) \lesssim \epsilon^3$$ as $\epsilon\to 0$.
\end{lem}
\begin{proof}
Since $|df|$ is uniformly bounded above and below on $\Omega_2$, it suffices to prove the same asymptotic for the area of $f(\Omega_2)\subset B_\epsilon$. Let $(u_1,\dots, u_n)$, $(x_1^2,x_2^2)$ be the coordinates from above, and let $Q$ be the embedding of the tangent plane $df(T_{z_2}\Sigma)$ inside our coordinate patch. Our metrics are locally comparable to Euclidean metrics, and hence 
$$\textrm{Area}(f(\Omega_2)) \lesssim \int_{\Omega_2} J(f)(x_1^2,x_2^2) dx_1^2d_x^2 \lesssim \int_{\Omega_2} J(f)(0) dx_1^2d_x^2 + O(\epsilon^3),$$ where $J(f)$ is the Jacobian determinant for $f$. Let $\tilde{\Omega}_2$ be the relevant component of $f^{-1}(C_\epsilon)$. Then $$\int_{\Omega_2} J(f)(0) dx_1^2dx_2^2 \leq \int_{\tilde{\Omega}_2} J(f)(0) dx_1^2dx_2^2.$$ As $f$ is an immersion near $z_2$, from multivariable calculus we have
\begin{equation}\label{28}
    \int_{\tilde{\Omega_2}} J(f)(0) dx_1^2dx_2^2= \int_{Q\cap C_\epsilon} dS,
\end{equation}
where $dS$ is the Euclidean area form on the parametrized surface $Q\cap C_\epsilon\subset \mathbb{R}^n$.

To compute, if $P$ and $Q$ span a four-dimensional subspace, then our job is very easy: the plane $Q$ only intersects $P$ at $f(z_2)$, and is hence contained in a ball of radius $\epsilon^{2}$. So the area integral is on the order of $\epsilon^{2n}$. The less trivial case is when $P$ and $Q$ intersect transversely inside a copy of $\mathbb{R}^3$. That is, $Q$ intersects $P$ in a line, making some acute angle $\alpha>0$ with $P$. The family of planes making such an angle admits an $S^1$ action by rotations around the normal axis (which we now assume is the $u_3$-axis), which preserves the area of intersections with the cylinder. Hence, we can replace $Q$ with any plane that makes the same angle $\alpha$. A convenient choice is $$Q = \{(u_1,u_2,u_3) : u_1\sin \alpha + u_3\cos \alpha = 0\}.$$ If not already the case, shrink $\epsilon$ so that $\epsilon<\tan \alpha$. We view $Q$ as the parametrized surface specified by $$F(u_1,u_2) = u_3 = -u_1\tan \alpha,$$ subject to the constraints $u_1^2+u_2^2<\epsilon^2$, $|u_3|<\epsilon^2$. Set $F_i = F_{u_i}$. We compute $$\int_{Q\cap C_\epsilon } dS = \int_{Q\cap C_\epsilon} \sqrt{F_1^2 + F_2^2 + 1} dS = 2\int_{-\frac{\epsilon^2}{\tan \alpha}}^{\frac{\epsilon^2}{\tan \alpha}}\int_0^{\sqrt{\epsilon^2 - u_1^2}} \sqrt{\tan^2 \alpha + 1} du_2 du_1 \lesssim \epsilon^3.$$ Thus, in both cases, inputting the estimates into (\ref{28}) gives the desired bound.
\end{proof}
With this lemma in hand, near $z_2$, $$X_2(z) = C\Big (\log \frac{1}{|z|}\Big )Z_2 + B(z)$$ with $Z_2$ a vector normal to $f(\Omega_2)$ at $z_2$ (and thus not normal to $P$) and $B\in C^{0,\alpha}(\mathbf{F}|_{\Omega_2})$. Independent of the choice of $\varphi_{\alpha\beta}$, we estimate 
$$\langle \mathcal{G}(\dot{\nu}),X\rangle = \sum_\gamma \mu^{ij}\dot{\Gamma}_{\alpha\beta,\gamma}\frac{\partial f^\alpha}{\partial x_i}\frac{\partial f^\beta}{\partial x_j}X^\gamma \lesssim \log \frac{1}{|x|}\max_{\alpha,\beta,\gamma}|\dot{\Gamma}_{\alpha\beta,\gamma}| \lesssim \epsilon^{-2}\log \frac{1}{|x|},$$ so that
    $$\int_{\Omega_2}  |\mathcal{G}(\dot{\nu}),X\rangle| dA \lesssim \epsilon^{-2} \int_{\Omega_2} \log \frac{1}{|x|} dA.$$ 
We bound the integral on the right:
\begin{align*}
    \int_{\Omega_2} \log \frac{1}{|x|} dA &= \int_{\Omega_2\backslash(\Omega\cap B(0,\epsilon^{3/2}))} \log \frac{1}{|x|} dA + \int_{\Omega \cap B(0,\epsilon^{3/2})}  \log \frac{1}{|x|} dA \\
    &\leq \log \epsilon^{-3/2}\textrm{Area}(\Omega_2) + \int_{ B(0,\epsilon^{3/2})}  \log \frac{1}{|x|} dA \lesssim \epsilon^3\log \epsilon^{-1}.
\end{align*}
Returning to our original integral, we obtain $$0=\int_\Sigma \langle \Jb V, X\rangle dA = \int_{\Omega_1} \langle \Jb V, X\rangle dA + O(\epsilon\log \epsilon^{-1}).$$ For $\Omega_1$, the area estimate
\begin{equation}\label{53}
    \epsilon^2\lesssim \int_{\Omega_1} dA \lesssim \epsilon^2
\end{equation}
is obvious. We choose $\varphi_{\alpha\beta}$ exactly as in the tangential case, and if $X$ does not extend smoothly over $z_1$, then using (\ref{53}) returns $$\int_{\Omega_1} \langle \Jb V, X\rangle dA \gtrsim \log \epsilon^{-1}$$ and produces the same contradiction as in the tangential case. This completes the proof of Lemma \ref{extend2}. As discussed above, this concludes our proof of transversality of $\Theta$.

\end{section}

\begin{section}{Immersions and Embeddings}
\begin{subsection}{Preparing the arguments}
Here we set up transversality arguments for Theorem B and Theorem C. We assume that $M$ is parallelizable, the general case being a slight modification (because transversality is a local property). Accordingly, we choose an isomorphism $\sigma: TM^\mathbb{C}\to M\times \mathbb{C}^n$ with projection map from $$TM^\mathbb{C}\to \mathbb{C}^n$$ that restricts to a family of isomorphisms $\sigma_p: TM_p^\mathbb{C}\to \mathbb{C}^n$, isometric with respect to the inner product induced by the metric on $TM_p^\mathbb{C}$ and the standard inner product on $\mathbb{C}^n$.

Let $\tilde{\Sigma}$ denote the universal cover of $\Sigma$. The metric $\mu$ on $\Sigma$ lifts to a metric on the universal cover $\tilde{\Sigma}$ that we still denote by $\mu$. Likewise, the harmonic map $f_{\mu,\nu}$ lifts to a map $f_{\mu,\nu}:(\tilde{\Sigma},\mu)\to (M,\nu)$, and we do not distinguish our notation.

The Riemannian manifold $(\tilde{\Sigma},\mu)$ identifies isometrically with $(\mathbb{D},\sigma)$, the complex unit disk endowed with its hyperbolic metric. We further identify the Riemann surface $\Sigma_\mu$ in the conformal class of $(\Sigma,\mu)$ with $\mathbb{D}/\Gamma_\mu$, where $\Gamma_\mu$ is a smoothly varying family of Fuchsian groups acting on $\mathbb{D}$. Let $z\in \mathbb{D}$ denote the complex parameter. This provides us with a canonical complex parameter $z_\mu=z$ on $\tilde{\Sigma}$ that depends only on $\mu$.

Unless stated otherwise, the dimension (codimension) of some object in a category (vector space, manifold, etc.) refers to the real dimension (codimension). To prove Theorem $B$, consider the subset $\mathcal{I}\subset \mathbb{C}^n$ defined by $$\mathcal{I}=\{A\in \mathbb{C}^n : \textrm{rank} A < 2\}.$$ Here, $\textrm{Rank}(A)$ denotes the dimension of the vector space spanned by $\textrm{Re}(A)$ and $\textrm{Im}(A)$. $\mathcal{I}$ is not a submanifold, but it splits as a union of two submanifolds of $\mathbb{C}^n$: $\mathcal{I}=\mathcal{L}_0\cup \mathcal{L}_1$, where $$\mathcal{L}_0=\{0\}\subset \mathbb{C}^n, \hspace{1mm} \mathcal{L}_1=\{A\in \mathbb{C}^n : \textrm{rank} A = 1\}\subset \mathbb{C}^n.$$ We define $$\Psi: \tilde{\Sigma}\times \M \to \mathbb{C}^n$$ by $$(p,\mu,\nu)\mapsto \sigma(f_z(p)).$$ The point is that $f_{\mu,\nu}$ is an immersion as long as $\Psi(p,\mu,\nu)\not \in \mathcal{I}$ for all $p$. $\mathcal{L}_0$ has codimension $2n$ and $\mathcal{L}_1$ has codimension $n-1$ in $\mathbb{C}^n$, so if $\Psi$ is transverse to both submanifolds, then $\Psi^{-1}(\mathcal{I})$ is contained in a codimension $n-1$ submanifold. Let $\pi: \tilde{\Sigma}\times \M\to \M$ be the projection map.  $\pi(\Psi^{-1}(\mathcal{I}))$ is contained in a submanifold of codimension at least $(n-1)-2 = n-3.$ So, once we formalize an argument using Proposition \ref{targ}, the content of Theorem B is that $\Psi$ is transverse at all points in the preimage of $\mathcal{I}$.

Toward transversality of $\Psi$, we compute the derivative of $\Psi$ at a point $(p,\mu,\nu)$ in the direction of a variation of the target metric $(0,0,\dot{\nu})$. As before, the infinitesimal variation of the maps $f_{\mu,\nu+t\dot{\nu}}$ is a section $V\in \Gamma(\mathbf{F})$ satisfying $\mathbf{J}V=\mathcal{G}(\dot{\nu})$. If we can choose our coordinate so that $V(p)=0$, then the vector $d\Psi(0,0,0,\dot{\nu})$ can be identified with the vertical lift of the associated vector in $\mathbb{C}^n$ under the identification of the tangent space at $0$. In such a coordinate, the derivative becomes $$d\Psi(0,0,\dot{\nu})=\sigma_p(\nabla_z V(p)).$$ Thus, we have the following. Recall that in Section 2 we defined the bundle $\mathbf{E}$ to be the complex bundle $f^*TM^{\mathbb{C}}.$
\begin{lem}
 Fix a point $(p,\mu,\nu)\in \tilde{\Sigma}\times \M$. Suppose that for every $W\in \mathbf{E}_p$, there exists a variation $\dot{\nu}\in T_{\nu}\M(M)$ such that if $V\in\Gamma(\mathbf{F})$ is the section satisfying $\mathbf{J}V = \mathcal{G}(\dot{\nu})$, then $V(p)=0$ and  $$\sigma_p(\nabla_z V(p)) = W(p).$$ Then $\Psi$ is a submersion at $(p,\mu,\nu)$. Consequently, $\Psi$ is transverse to $\mathcal{L}_0$ and $\mathcal{L}_1$ at $(p,\mu,\nu)$.
\end{lem}
Granting the following, we prove Theorem B.
\begin{lem}\label{trans1}
Suppose $f_{\mu,\nu}$ is somewhere injective and has isolated singularities. Then the hypothesis of the lemma above is satisfied. Hence, $\Psi$ is transverse to $\mathcal{L}_0$ and $\mathcal{L}_1$ at $(p,\mu,\nu)$.
\end{lem}
\begin{proof}[Proof of Theorem B]
Transversality is an open property, so we can fix a neighbourhood $U$ around $(\mu,\nu)$ in which $\Psi$ is transverse. We let $U^I$ denote the subset of $(\mu,\nu)\in U$ corresponding to harmonic immersions. Observe $U^I = U\cap\M\backslash(\pi(\Psi^{-1}(\mathcal{I})))$. The goal is to show this is open, dense, and connected.

$\mathcal{I}$ is clearly closed, from which the openness result is immediate. For density and connectedness, we use Proposition \ref{targ} with $A=U^I$, $X=\tilde{\Sigma}$, $Y=\mathbb{C}^n$, $f=\Psi$ and $W$ being both $\mathcal{L}_0$ and $\mathcal{L}_1$. Note that we computed the derivative of $\Psi$ above, so it is clearly $C^1$ (in fact, it is $C^m$, but we don't need to prove this). $d=\dim X=2$ and $k=\textrm{codim}_{\mathbb{C}^n}W=2n$ and $n-1$ for $W=\mathcal{L}_0$ and $\mathcal{L}_1$ respectively. Hence, if $n=\dim M \geq 4$, then in both cases $d-k\leq 2-3=-1,$ so $U^I$ is dense. If $n=\dim M\geq 5,$ then $d-k\leq -2,$ so $U^I$ is connected.
\end{proof}
We now explain Theorem C. Define $$\Phi: \Sigma^2\times \M \to M^2$$ by $$(p,q,\mu,\nu)\mapsto (f_{\mu,\nu}(p),f_{\mu,\nu}(q)),$$ and let $\mathcal{E}$ be the diagonal $$\mathcal{E}=\{(x,x): x\in M\}\subset M^2.$$ Similar to before, the bulk of the proof consists of showing that $\Phi$ is transverse to $\mathcal{E}$ at certain points. The derivative in a $\dot{\nu}$ direction is just $(V(p),V(q))$, where $V$ is the associated harmonic variation (Section 5.1). The following lemma is the transversality criterion.
\begin{lem}
Fix points $(p,q,\mu,\nu)\in\Sigma^2\times \M$. Suppose that for every $W_1\in \mathbf{F}_p$, $W_2\in \mathbf{F}_q$, there exists a variation $\dot{\nu}\in T_{\nu}\M(M)$ such that if $V\in\Gamma(\mathbf{F})$ is the section satisfying $\mathbf{J}V = \mathcal{G}(\dot{\nu})$, then $(V(p),V(q))=(W_1,W_2)$. Then $\Psi$ is a submersion at $(p,q,\mu,\nu)$. Consequently, $\Psi$ is transverse to $\mathcal{E}$ at $(p,q,\mu,\nu)$.
\end{lem}
As above, Theorem C follows from a lemma that we leave for later.
\begin{lem}\label{trans2}
Suppose $f_{\mu,\nu}$ is somewhere injective and has isolated singularities. Then the hypothesis of the lemma above is satisfied. Hence, $\Psi$ is transverse to $\mathcal{E}$ at $(p,\mu,\nu)$.
\end{lem}
Assuming this lemma, the proof of Theorem C follows the same line as the proofs of Theorems A and B.
\begin{proof}[Proof of Theorem C]
    Similar to above, we fix a neighbourhood $U$ around $(\mu,\nu)$ in which $\Phi$ is transverse, and we let $U^E\subset U$ be the subset corresponding to harmonic embeddings. $U^E$ is open for basic reasons. To prove $U^E$ is dense for $n\geq 5$ and connected for $n\geq 6$, we use Proposition \ref{targ} with $A=U^E,$ $X=\Sigma^2,$ $Y=M^2,$ $f=\Phi,$ and $W=\mathcal{E}$. $d=\dim \Sigma^2=4$ and $k = \textrm{codim}_{M^2} \mathcal{E} = n.$ If $n=\dim M\geq 5,$ then $d-k\leq 4-5=-1,$ so $U^E$ is dense. If $n=\dim M\geq 6,$ then $d-k\leq -2,$ so $U^E$ is connected.
\end{proof} 
\end{subsection}

\begin{subsection}{The holomorphic line bundle $\mathbf{L}$} Working toward the lemmas, we introduce the line bundle $\mathbf{L}$. Here we follow the exposition of \cite[section 4.1]{M2}. Fix a pair $(\mu,\nu)\in \mathfrak{M}$ and let $f=f_{\mu,\nu}$ denote the associated harmonic map. As in Section $2$, if we take a local complex parameter $z=x+iy$, $f_z$ is a local holomorphic section of the bundle $\mathbf{E}$, which we recall is equipped with its Koszul-Malgrange holomorphic structure.

While $f_z$ is only locally defined, the zero set is independent of the choice of coordinate, and the projectivization $[f_z]$ is a well-defined holomorphic section of the projectivized bundle $\mathbb{P}(\mathbf{E})$. Analytically continuing to the zero set we obtain a well-defined global section $$[f_z]:\Sigma\to \mathbb{P}(\mathbf{E}).$$ This section defines a family of lines in $\mathbf{E}$, which patch together to form a holomorphic line bundle $\mathbf{L}\subset \mathbf{E}$, and $f_z$ may be naturally viewed as a local holomorphic section of $\mathbf{L}$. If $p$ is a branch point, we can choose a coordinate $z$ in which $z(p)=0$ and $$f_z=z^kg(z)$$ where $g$ is a local section of $\mathbf{L}$ such that $g(p)\neq 0$. The integer $k$ is called the branching order of $f$.

The $\mathbf{E}$-valued $(1,0)$-form $f_z dz$ is naturally a holomorphic section of the holomorphic vector bundle $\mathbf{L}\otimes \mathbf{K}$, where $\mathbf{K}$ is the canonical bundle. If $f$ branches at points $p_1,\dots, p_n$ with branching orders $k_{p_1},\dots, k_{p_n}$, then $f_zdz$ defines a nowhere vanishing holomorphic section of the bundle $$\mathbf{L}\otimes \mathbf{K}\otimes \zeta_{p_1}^{-k_{p_1}}\otimes \dots \otimes \zeta_{p_n}^{-k_{p_n}}$$ where $\zeta_{p_j}$ is the holomorphic point bundle at $p_j$. It follows that $$\mathbf{L}\simeq \mathbf{K}^*\otimes \zeta_{p_1}^{k_{p_1}}\otimes \dots \otimes \zeta_{p_n}^{k_{p_n}}.$$
The degree of $\mathbf{L}$ can then be computed by the evaluation of the first Chern class against the fundamental class of $\Sigma$: $$\deg L=\langle c_1(\mathbf{L}), [\Sigma]\rangle = 2-2g + \sum_p k_p,$$ where the sum is taken over the branch set.

\end{subsection}

\begin{subsection}{Prescribing harmonic variations for Lemma \ref{trans1}}
Lemma \ref{trans1} is a special case of the following stronger result.
\begin{lem}
Fix a local complex coordinate $z=x+iy$ near $p\in \Sigma$. Then, for any three vectors $Z_j\in \mathbf{F}_p$, $j=1,\dots, 3$, we can find $\dot{\nu}\in T_\nu\mathfrak{M}^*(M)$ such that 
$$V(p)=Z_1 \hspace{1mm} , \hspace{1mm} \nabla_x V(p)=Z_2 \hspace{1mm} , \hspace{1mm} \nabla_y V(p)=Z_3$$
where $\mathcal{G}(\dot{\nu})=\mathbf{J}V$.
\end{lem}
Suppose the lemma false, so that there are three vectors $Z_1,Z_2, Z_3\in \mathbf{F}_p$ such that the above fails for every $V$ of the form $\mathbf{J}V=\mathcal{G}(\dot{\nu})$, where $\dot{\nu}\in T_\nu\mathfrak{M}(M)$. Considering the induced inner product on $\oplus_1^3\mathbf{F}_p$, we can find a triplet of vectors $U_1,U_2, U_3\in \mathbf{F}_p$ (with not all of them equal to the zero vector) such that for every section $V\in \Gamma(\mathbf{F})$ such that $\mathbf{J}V=\mathcal{G}(\dot{\nu})$, $U_1$, $U_2,$ and $U_3$ are all orthogonal to $V(p)$. This yields the identity $$\langle V(p),U_1\rangle + \langle \nabla_x V(p), U_2\rangle + \langle \nabla_y V(p), U_3\rangle = 0$$ for every such $V$.

Adding together reproducing kernels, we obtain a smooth section $X:\Sigma\backslash\{p\}\to \mathbf{F}$ such that
\begin{equation}\label{notzero}
    \langle W(p),U_1\rangle + \langle \nabla_x W(p), U_2\rangle + \langle \nabla_y W(p), U_3\rangle =
\int_\Sigma \langle \mathbf{J}W, X\rangle dA
\end{equation}
for every $W\in \Gamma(\mathbf{F})$. Moreover, $\mathbf{J}X(X)=0$ for every $x\in \Sigma\backslash\{p\}$ and the growth of $X$ is controlled by $|z|^{-1}$ at $p$.
\begin{lem}
$X$ is not identically zero.
\end{lem}
\begin{proof}
It is an elementary exercise to show that one can construct sections of $\mathbf{F}$ with prescribed $1$-jet at $p$ (and we used this fact already in Section $3$). So, one can choose $W$ such that the left-hand side of equation (\ref{notzero}) is positive.
\end{proof}

The following lemma, a very important piece of our argument, is the content of \cite[Lemma 3.1]{M1}. The argument can also be found in Moore's book \cite[page 311]{M2}.
\begin{lem}\label{Moorelemma}
Let $\Omega\subset \Sigma^{reg}(f)$ be a small open subset of the regular set of $f$, and assume $f=f_{\mu,\nu}$ satisfies $f^{-1}(f(\Omega))=\Omega$. Suppose $Y:\Omega\to \mathbf{F}$ is a smooth section. If $$\int_\Sigma \langle \mathbf{J}V, Y\rangle dA = 0$$ for every $\dot{\nu}\in T_\nu\mathfrak{M}(M)$ whose support is contained in $f(\Omega)$, then each point $p\in \Omega$ has a neighbourhood on which $Y$ equals the real part of a local holomorphic section of $\mathbf{L}$.
\end{lem}
\begin{remark}
The existence of such a set $\Omega$ is guaranteed by the hypothesis that $f$ is somewhere injective.
\end{remark}
Since the somewhere injective property is so strongly used, we give a word on the proof. 
\begin{proof}[Ideas in the proof]
The hypothesis that $f^{-1}(f(\Omega))=\Omega$ implies that if we take any variation $\dot{\nu}$ with support in $f(\Omega)$, then the support of $\mathbf{J}V=\mathcal{G}(\dot{\nu})$ is contained in $\Omega$, where $V$ is the associated harmonic variation. Therefore,
\begin{equation}\label{localint}
    \int_\Sigma \langle \mathbf{J}V, Y\rangle dA = \int_\Omega \langle \mathbf{J}V, Y\rangle dA.
\end{equation}
Choosing variations $\dot{\nu}$ normal to $f(\Omega)$, Moore uses (\ref{localint}) to show that $Y$ is a tangential section of $\mathbf{F}$ over $\Omega$, i.e., it maps into the image of $df(T\Sigma|_U)$ inside the pullback bundle $f^*TM$. Since $f$ is regular in $\Omega$, one can identify $\mathbf{F}|_\Omega$ with a real subbundle of $\mathbf{L}$. This will be explained after Proposition \ref{Xperp}. Then, choosing tangential variations, (\ref{localint}) is used to show that, under the identification, $Y$ is the real part of a holomorphic section of $\mathbf{L}$.
\end{proof}

Now we return to our main argument. Choose an open set $\Omega$ as above and not containing $p$ and apply Lemma \ref{Moorelemma} to the section $X$. Note that $X$ has no singularity in $\Omega$. Let $Z$ be the holomorphic section of $\mathbf{L}$ defined on $\Omega$ corresponding to $X$.

We use the isolated singularity condition to analytically continue $Z$. Let $U$ be any open subset of the regular set that intersects $\Omega$ with non-empty interior. We chose a conformal coordinate $z=x_1+ix_2$ on the source as well as coordinates on the target so that we could write $$X=X^j\frac{\partial}{\partial u_j}$$ with $\partial f^i/\partial x_j=\delta_{ij}u_i$. This identifies the first two components with the tangent bundle over $\Omega$, and we get an orthogonal splitting into tangential and normal components as $$f^*TU = (f^*TU)^T \oplus (f^*TU)^\perp.$$ $X|_\Omega$ is tangential, and hence the projection $\pi_{(f^*TU)^\perp}(X)$ vanishes on $\Omega$. We see via the next proposition that this holds on all of $U$.
\begin{prop}\label{Xperp}
$\pi_{(f^*TU)^\perp}(X) = 0$ on all of $U$. In other words, $X|_U$ is tangential to the image surface $f(U)$.
\end{prop}
\begin{remark}
This is automatic when the metrics $(\mu,\nu)$ are real analytic (which implies $f$ is real analytic as well).
\end{remark}
\begin{proof}
We prove $\pi_{(f^*TU)^\perp}(X) = 0$ in an open disk $V\subset U$ that intersects $\Omega$. The proposition then follows from point-set considerations. 

Let $\Big \{\frac{\partial}{\partial f_1},\frac{\partial}{\partial f_2},\dots \frac{\partial}{\partial f_n}\Big \}$ be a trivialization for $f^*TV$ such that $\Big \{\frac{\partial}{\partial f_1},\frac{\partial}{\partial f_2}\Big \}$ and $\Big \{\frac{\partial}{\partial f_3},\dots \frac{\partial}{\partial f_n}\Big \}$ are frames for the tangential and normal subbundles respectively. In these frames, write $X=X^j\frac{\partial}{\partial f_j}$, so that $$\pi_{(f^*TU)^\perp}(X) = \sum_{j=3}^n X^j \frac{\partial}{\partial f_j}.$$ Let $p\in V$ and let $z=x+iy$ be a local complex parameter for $V$ with $z(p)=0$. In the coordinate, $\mathbf{J}X$ is expressed as 
\begin{align*}
    \mathbf{J}X &=  -\mu^{-1} \Big(\sum_{j=1}^n (X_{xx}^j+X_{yy}^j)\frac{\partial}{\partial f_j}\Big ) + \mu^{-1}\sum_{j=1}^n \Big (2X_x^j \nabla_x \partial f_j - 2X_y^j\nabla_y \frac{\partial}{\partial f_j}\Big ) \\ 
    &-\mu^{-1}\Big (\sum_{j=1}^n X^j(\nabla_x\nabla_x + \nabla_y \nabla_y) \frac{\partial}{\partial f_j}
    -\mu^{-1}\sum_{i,k,j,\ell}X^k ({}^\nu R_{\ell k i}^j\circ f)(f_x^\ell f_x^i+f_y^\ell f_y^\alpha)\frac{\partial}{\partial f_j}\Big ),
\end{align*}
where the ${}^\nu R_{\delta\beta\alpha}^\gamma$ are the coordinate expressions for the Riemannian curvature tensor of $(M,\nu)$. Since $\mathbf{J}X=0$ on $V$, we deduce that for all $j$, $$\Big | \Big (\frac{\partial^2}{\partial x^2} + \frac{\partial^2}{\partial y^2} \Big ) X^j \Big | \lesssim |\nabla X| + |X|,$$ where $\nabla$ is the ordinary Euclidean gradient in the local coordinates. We can now invoke the Hartman-Wintner theorem \cite{HW}, which asserts that, in our choice of coordinates, $$X(z) = p(z) + r(z),$$ where $p(z)$ is a vector-valued harmonic homogeneous polynomial, and $r(z)\in O(zp(z))$. It follows immediately that $X^j=0$ on $V$ for $j\geq 3$.
\end{proof}

Next, let $\gamma$ be any curve emanating from $\Omega$ that does not intersect the singular set (this includes $p$). In a neighbourhood $U$ containing the first intersection point of $\gamma\cap \Omega$, we continue to choose a conformal coordinate $z=x_1+x_2$ so that $\partial f^i/\partial x_j=\delta_{ij}u_i$. Then there is a real linear isomorphism $\tau: (f^*TU)^T\to \mathbf{L}|_U$ defined by
\begin{equation}\label{liniso}
    M\partial/\partial x_1 + N\partial/\partial x_2\mapsto (M+iN)f_z.
\end{equation}
In $\Omega\cap U$, the proof of Lemma \ref{Moorelemma} explicitly constructs the holomorphic section $Z$ as $$Z=\tau(X) = (X_1+iX_2)\frac{\partial}{\partial z}.$$ Extending $Z$ by this formula on all of $U$, it is easily checked that $Z$ is a Jacobi field if and only if $X$ is. Thus, arguing similarly to above, we see via Hartman-Wintner that in our coordinates, $$Z(z)=q(z)+s(z),$$ with $q$ a complex vector-valued harmonic homogeneous polynomial and $s(z)$ decaying faster. Differentiating, the local expression for the section $\overline{\partial}Z$ takes this form as well. Thus, since $\nabla_{\overline{z}}Z =\overline{\partial}Z$ vanishes on $\Omega$, it vanishes everywhere. That is, $Z$ is holomorphic on $U$. In this way, we continue along all of $\gamma$. The next lemma shows that the analytic continuation does not depend on the path. 
\begin{lem}
Let $W$ denote a local holomorphic section of $\mathbf{L}$. Then $\textrm{Re}(W)$ is not identically $0$.
\end{lem}
Indeed, suppose we have two open sets $\Omega_1,\Omega_2\supset \Omega$ and local holomorphic extensions $X_1,X_2$ of $X$. Setting $W=X_1-X_2$, the lemma above forces $W\equiv 0$. Thus $Z$ extends in well-defined fashion to the complement of the singular set.
\begin{proof}
This is also found in \cite[Proposition 4.2]{Ma}, but we include the proof for completeness. In a local complex parameter $z=x+iy$, $W$ may be written $W=hf_z$ for some locally defined meromorphic function $h=h_1+ih_2$ (with possible poles matching up with zeros of $f$). Then $$\textrm{Re}(W)=\frac{1}{2}(h_1f_x+h_2f_y).$$ If $W$ is non-zero and $\textrm{Re}(W)\equiv 0$, then $f_x$ and $f_y$ are linearly dependent vectors, and moreover $\rank(df)<2$ on $\Omega$. This is impossible since, as remarked earlier, the set of regular points for $f$ is open and dense in $\Sigma$.
\end{proof}
We now address singular points.
\begin{lem}\label{extendsin}
$Z$ extends holomorphically over every singular point except possibly $p$.
\end{lem}
\begin{proof}
Let $\mathcal{S}$ be the singular set of $f$, so that we have a section $Y:\Sigma\backslash\mathcal{S}\to \mathbf{F}$ such that $Z=X+iY:\Sigma\backslash\mathcal{S}\to \mathbf{L}$ is holomorphic.

From the local coordinate expression for $Z$, the norm with respect to the natural metric on $\mathbf{F}$ blows up at singular points at worst like the inverse of the Jacobian of $f$. Thus, $Z$ extends to a meromorphic section of $\mathbf{E}$ on all of $\Sigma$. Taking the projectivization gives a well-defined holomorphic section $[Z]:\Sigma\to \mathbb{P}(\mathbf{E})$, which by the identity theorem must agree with $[f_z]$. That is, $Z$ is parallel to $f_z$, even at the singularities.

Let $q\neq p$ be a singularity. Choosing a local complex coordinate $z$ with $z(q)=0$ and any trivialization for our bundle, we can write $$Z=z^{n}g(z),$$ with $g$ holomorphic and parallel to $f_z$, $g(0)\neq 0$, and $n\in\mathbb{Z}$. Write $g=h(z)f_z$, with $h$ meromorphic. Let $g=g_1+ig_2$, $h^{-1}=h_1^{-1}+ih_2^{-1}$. Then $$f_z = \frac{1}{2}(f_x - if_y) =\frac{1}{2}\Big ( (h_1^{-1}g_1 + h_2^{-1}g_2) - i(h_2^{-1}g_1-h_1^{-1}g_2) \Big ),$$ and if $X=a(z)f_x+b(z)f_y$ away from $q$ (the coefficients may blow up at $q$), then we can also write
\begin{align*}
    X&= a(z)(h_1^{-1}g_1 + h_2^{-1}g_2)  + b(z)(h_2^{-1}g_1-h_1^{-1}g_2) \\
    &= (a(z)h_1^{-1}+b(z)h_2^{-1})g_1 + (a(z)h_2^{-1}-b(z)h_1^{-1})g_2.
\end{align*}
Since $X$ is regular and bounded at $q$, the coefficents on $g_1$ and $g_2$ are regular and bounded. This demonstrates that $X$ can be expressed as a real section in the trivialization determined by $g$. The same can be done for $Y$ off $q$. Since $X$ is bounded at $q$, it follows that the singularity of $Z$ is removeable.
\end{proof}

To obtain a contradiction and finish the proof of Lemma \ref{trans1}, we explain that no holomorphic section such as $Z$ can exist. From the proof of Lemma \ref{extendsin}, $Z$ behaves at worst like the asymptotic (\ref{overz}), so it extends to a globally defined meromorphic section $$Z:\Sigma\to \mathbf{L}$$ with a pole of order at most $1$ at $p$. We let $\textrm{ord}_q^{\mathbf{L}}(X)$, $\textrm{ord}_q^{\mathbf{E}}(X)$  denote the order of vanishing of $Z$ at $q$ with respect to the charts on the bundles $\mathbf{L}$ and $\mathbf{E}$ respectively. In this notation, $$\textrm{ord}_{p_j}^{\mathbf{L}}(Z)= \textrm{ord}_{p_j}^{\mathbf{E}}(Z)+k_{j}.$$ The degree of the divisor for $Z$ with respect to $\mathbf{L}$ agrees with the degree of $\mathbf{L}$, so that $$\deg \mathbf{L}= \sum_{q\in \Sigma} \textrm{ord}_q^{\mathbf{L}}(Z)= \sum_{q\in \Sigma} \textrm{ord}_q^{\mathbf{E}}(Z) + \sum_j k_j\geq -1 + \sum_j k_j.$$ Meanwhile, we showed earlier that $$\deg \mathbf{L}=2-2g+\sum_j k_j.$$ This implies $2-2g \geq -1$, or $g\leq 3/2$, and this contradiction establishes the result.

\end{subsection}

\begin{subsection}{Harmonic embeddings: the proof of Lemma \ref{trans2}} This is similar to Lemma \ref{trans1}, so we only sketch the proof. Lemma \ref{trans2} amounts to proving the following statement.
\begin{lem}
For any two vectors $Z_1\in \mathbf{F}_p$ and $Z_2\in \mathbf{F}_q$, we can find $\dot{\nu}\in T_\nu\mathfrak{M}^*(M)$ such that 
$$V(p)=Z_1 \hspace{1mm} , \hspace{1mm} V(q)=Z_2,$$
where $\mathcal{G}(\dot{\nu})=\mathbf{J}V$.
\end{lem}
If the lemma fails, there are vectors $U_1\in\mathbf{F}_p$, $U_2\in\mathbf{F}_q$ such that for every harmonic variation $V$, $$\langle U_1, V(p)\rangle + \langle U_2, V(q)\rangle =0.$$ Taking the reproducing kernels for $U_1$ and $U_2$, we have a section $X:\Sigma\backslash\{p,q\}\to \mathbf{F}$ such that $$\int_\Sigma \langle \mathbf{J}W, X \rangle dA = \langle U_1, W(p)\rangle + \langle U_2, W(q)\rangle$$ for all sections $W\in \Gamma(\mathbf{F})$. Invoking Moore's lemma and then repeating our argument from the previous subsection, one finds a section $Y:\Sigma\backslash\{p,q\}\to \mathbf{F}$ such that $Z=X+iY$ is holomorphic. The asymptotic (\ref{log}) ensures that $Z$ blows up strictly slower than any meromorphic section, and hence the singularities at $p$ and $q$ are removeable. Thus, $Z$ yields a globally defined holomorphic section, a Jacobi field, and this is a contradiction.
\end{subsection}

\begin{subsection}{Toward the Whitney theorems}
Theorems B and C show that Question E has a positive answer near harmonic surfaces that are somewhere injective and have isolated singularities. To conclude the paper, we discuss our use of this hypothesis and the possibility of removing it.

Firstly, to positively resolve Question E, by Theorem A it suffices to prove it is true near surfaces that are somewhere injective. So the only extra condition that we use here is the isolated singularities. Let us assume that $(\mu,\nu)$ are such that $f_{\mu,\nu}$ is somewhere injective, with no condition on singularities. Beginning the proof of Lemma \ref{trans1}, we find a kernel $X$ satisfying (\ref{notzero}). Then we can find an open set $\Omega$ on which $f$ is injective, and by Moore's lemma, $X|_\Omega$ is the real part of a holomorphic section $Z:\Omega\to \mathbf{L}$. At this point, it is tempting to believe that some sort of unique continuation argument should promote this to a global result.

It is unclear if this is possible, one obstruction being that we are not aware of a coordinate-free way to express that $X$ is the real part of a holomorphic section. The issue stems from the following remark: $X$ is a real section of $\mathbf{E}$ with respect to the real structure induced from the splitting $\mathbf{E}=\mathbf{F}\oplus i \mathbf{F}$, but there is no reason for a transition map to the holomorphic trivialization for the Koszul-Malgrange holomorphic structure to preserve this real structure. All we can say is that in such a trivialization, $$X(z) = K(z)X_0(z),$$ where $X_0(z)$ is a real section, and $K(z)$ is a smoothly varying family of complex matrices. Furthermore, the imaginary part $Y$ has been defined in terms of a particular local frame for $\mathbf{F}$. And the norm of the elements in this frame may explode as we approach singularities of $f$. In other words, we have no apriori uniform continuity for $Z$ in $\Omega$, and the imaginary part could blow up in an attempt to analytically continue.

While the holomorphic coordinates on $\mathbf{E}$ are opaque, we do have some understanding of what it means to be a holomorphic section of $\mathbf{L}$. This is what allows for some results under stronger hypothesis. In the end we want to find our contradiction by realizing $X$ as the real part of a global meromorphic section of $\mathbf{L}$ with constrained poles. Two steps:
\begin{enumerate}
    \item Show that when $f$ is regular, $X$ is tangential to $f$.
     \item Find mappings from the distribution $df(T\Sigma)\subset f^*TM$ to $\mathbf{L}$, under which $X$ corresponds to the real part of a meromorphic section (a meromorphic multiple of $f_z$).
\end{enumerate}
When $f$ is regular, we have well-defined splittings of $f^*TM$ into tangential and normal components for the image of $f$. Thus, if the set $A(f)$ from Section $4$ is connected, we can show (1) holds via a unique continuation argument (Proposition \ref{Xperp}). Note that the isolated singularity condition is really more than what we need for this to work. Once we have (1), we can define the section $Y$ at regular points as before, and again a connectedness assumption allows us to deduce that $Z=X+iY$ is holomorphic where defined.

The last challenge is to extend $Z$ over singular points. If the singularity is isolated, then $Z$ is meromorphic, and then we can argue using the boundedness of $X$ in the right choice of coordinates. But with a more complicated singular set--say, a general fold or a meeting point of general folds--controlling $Z$ becomes a delicate task. 

At this point, we see no direct geometric reason for the argument to fail in general. It is reasonable to expect that we can relax our assumptions to include harmonic maps with particular types of (non-isolated) singularities. It is unclear how far the method goes.
\end{subsection}

\end{section}

\appendix

\section{Transversality theorems}
We explain the transversality theory used in the proofs of Theorems A, B, and C. The main result is Proposition \ref{targ}. For more background on transversality theory and Banach manifolds in general, we refer the reader to \cite{ARflows}, \cite[Chapter 1]{M2}, and \cite{Smale}.

\begin{defn}
Let $X,Y$ be $C^1$ manifolds, $f:X\to Y$ a $C^1$ map, and $W\subset Y$ a submanifold. We say that $f$ is transverse to $W$ at a point $x\in X$ if $f(x)=y\not\in W$ or if $f(x)=y\in W$ and 
\begin{itemize}
    \item the inverse image $(df_x)^{-1}(T_yW)$ splits and 
    \item the image $df_x(T_xX)$ contains the closure of the complement of $T_yW$ in $T_yY$.
\end{itemize}
We say $f$ is transverse to $W$ if we have transversality for every $x\in X$.
\end{defn}
The central transversality theorem is below.
\begin{thm}[Transversality Theorem for Banach Manifolds] Let $X,Y$ be $C^r$ manifolds ($r\geq 1$), $f:X\to Y$ a $C^r$ map, and $W\subset Y$ a $C^r$ submanifold. Then if $f$ is transverse to $W$,
\begin{itemize}
    \item $f^{-1}(W)$ is a $C^r$ submanifold of $X$ and 
    \item if $W$ has finite codimension in $Y$, then $\codim_X f^{-1}(W)=\codim_Y W$.
\end{itemize}
\end{thm}

A $C^1$ map between $C^1$ manifolds $f:X\to Y$ is Fredholm if for every $x\in X,$ the tangent map $df_x: T_xX \to T_{f(x)}Y$ is a Fredholm operator. If $X$ is connected, then the index, which does not depend on the point $x$, is $\textrm{index}(f)=\dim \ker (df_x)-\textrm{codim}_{T_{f(x)}Y}(df_x(T_xX)).$ 

\begin{thm}[Smale's Density Theorem] Let $X$ and $Y$ be $C^r$ manifolds with $X$ Lindel{\"o}f and $f:X\to Y$ a $C^r$ Fredholm map. Suppose that $r>\max\{0,\textrm{index}(f)\}$. Then the set of regular values of $f$ is residual.
\end{thm}

Recall that a subset of a topological space is residual if it is a countable intersection of dense open subsets. By the Baire Category theorem, residual sets are dense. Note that second countable implies Lindel{\"o}f.

The two stated results are sufficient to prove the result below that is used throughout the paper.
\begin{prop}\label{targ}
    Let $A,X,Y$ be connected $C^1$ Banach manifolds, and $f: A\times X\to Y$ a $C^1$ map. Let $W\subset Y$ be a $C^1$ submanifold. Assume that
    \begin{itemize}
        \item $X$ has finite dimension $d$ and $W$ has codimension $k$ in $Y$,
        \item $A$ and $X$ are second countable, and
        \item $f$ is transverse to $W$.
    \end{itemize}
    Let $\pi:A \times X\to A$ be the projection map and $Z=\pi(f^{-1}(W))$.
    \begin{itemize}
        \item If $d-k\leq -1,$ then $A\backslash Z$ is dense, and
        \item if $d-k\leq -2,$ then $A\backslash Z$ is connected.
    \end{itemize}
\end{prop}
In the proof of Proposition A.4, we essentially go through the argument used to prove the Parametric Transversality Theorem. For completeness, we state the theorem and give some explanation. Let $A,X,Y$ be $C^r$ manifolds and $\delta: A\to C^r(X,Y)$ a map. The evaluation map  $\beta:A \times X\to Y$ is defined by $\beta(a,x) = \delta(a)(x)$.

\begin{thm}[Parametric Transversality Theorem] Let $A,X,Y$ be $C^r$ manifolds and $\delta: A \to \mathcal{C}^r(X,Y)$ a map such that the evaluation map $\beta$ is $C^r$. Let $W\subset Y$ be a $C^r$ submanifold and $A_W$ the set of $a\in A$ such that $\delta(a)$ is transverse to $W$. Assume that
\begin{itemize}
    \item $X$ has finite dimension $d$ and $W$ has finite codimension $k$ in $Y$,
    \item $A$ and $X$ are second countable,
    \item $r>\max(0,d-k)$, and 
    \item the evaluation map is transverse to $W$.
\end{itemize}
Then $A_W$ is residual in $A$.
\end{thm}
The Parametric Transversality Theorem is typically proved as follows. By transversality, $\beta^{-1}(W)$ is a submanifold of codimension $k$. The second countable assumption on $A$ and $X$ guarantees that $\beta^{-1}(W)$ is Lindel{\"o}f. If $\pi:A\times X\to A$ is the projection map, then one observes that the restricted map $\pi|_{\beta^{-1}(W)}:\beta^{-1}(W)\to A$ is Fredholm of index $d-k$ \cite[Lemma 19.2]{ARflows}. One then proves that $A_W$ is the residual set of $\pi|_{\beta^{-1}(W)}$ \cite[Lemma 19.4 and Lemma 19.5]{ARflows} and applies Smale's Density Theorem to $\pi|_{\beta^{-1}(W)}.$  

We also point out that parametric transversality recovers Smale's Density Theorem. Suppose that $f:A\to Y$ is a $C^r$ Fredholm map with $r>\max \{0,\textrm{index}(f)\}$. Setting $X=\{0\}$ and defining $\delta:A\to C^r(\{0\},Y)$ by $\delta(a)(0)=f(a)$, the Parametric Transversality Theorem applied to $\delta$ yields Smale's Density Theorem for $f$.  With all of this explained, we give the proof of Proposition \ref{targ}.
\begin{proof}[Proof of Proposition \ref{targ}]
    By the Transversality Theorem for Banach Manifolds, $f^{-1}(W)\subset A\times X$ is a submanifold of codimension $k$, which is Lindel{\"o}f by our second countable assumption on $A$ and $X$. The projection map $\pi|_{f^{-1}(W)}: f^{-1}(W)\to A$ is $C^1$ and Fredholm of index $d-k$ \cite[Lemma 19.2]{ARflows}. If the index of $\pi|_{f^{-1}(W)}$ is negative, $a\in A$ is a regular value if and only if $a\in A\backslash Z$. Assuming $d-k\leq -1,$ it follows from Smale's Density Theorem that $A\backslash Z$ is dense. 

    Now assume $d-k\leq -2.$ Let $\gamma:[0,1]\to A$ be any $C^1$ path with endpoints in $A\backslash Z$. We perturb $\gamma,$ while keeping the endpoints fixed, so that its image is contained in $A\backslash Z.$ First we partition $[0,1]$ into sufficiently small intervals $I_0,\dots, I_n$, such that the image of $I_i$ under $\gamma$ is contained in a sufficiently small subset $U_i$ of $A$ that fits into a single (convex) chart for the model Banach space of $A$. 

    Let $a_0\in A\backslash Z$ be the starting point of $\gamma$ and $a_1\in U_0\cap U_1$. Working in the chart for $U_0,$ we define a map $$\delta: U\to C^1([0,1]\times X, Y)\hspace{1mm}, \hspace{1mm} \delta(a)(t,x)=f(a_0+ta,x),$$ whose evaluation map is clearly $C^1$. Since $\dim ([0,1]\times X)=d+1$ and $\textrm{codim}_Y(W)=k$, we have that $d+1-k\leq -1$. Since $f$ is transverse to $W$, the associated evaluation map $\beta$ is transverse to $W$. The set of $a\in U_0$ such that $\delta(a)$ is transverse to $W$ is the set of regular values of the projection from $U\times ([0,1]\times X)\to U,$ restricted to $\beta^{-1}(W)$ \cite[Lemma 19.4 and Lemma 19.5]{ARflows}. By Smale's Density Theorem, we obtain that on a dense set of $a\in U_0,$ the image $\delta(a)([0,1]\times X)$ does not intersect $W$. We can thus choose $a\in U_0\backslash Z$ arbitrarily close to $a_1$ and so that the straight line path connecting $a_0$ to $a$ is contained in $U\backslash Z$. Henceforth replace $a_1$ with $a$.

    Returning to the path $\gamma,$ for $i=2,\dots, n,$ choose points $a_i\in U_{i-1}\cap U_i$, and let $a_{n+1}$ be the endpoint of $\gamma$. Choose a path from $a_1$ to $a_2$ contained entirely $U_1$. We repeat the above procedure to obtain an arbitrarily close path contained entirely in $U_1\backslash Z.$ Continuing inductively and concatenating our paths, we obtain the desired path from $a_0$ to $a_{n+1}.$
\end{proof}

\bibliographystyle{plain}
\bibliography{bibliography}

\end{document}